\documentclass[a4paper,11pt,twoside]{amsart}

\usepackage{graphicx}
\usepackage{amsfonts}
\usepackage{amsmath}
\usepackage{amsthm}
\usepackage{amssymb}
\usepackage[all]{xy}
\usepackage{hyperref}
\usepackage{rotating}

\newtheorem{thm}{Theorem}[section]
\newtheorem{prop}[thm]{Proposition}
\newtheorem{lem}[thm]{Lemma}
\newtheorem{cor}[thm]{Corollary}
\newtheorem{conjecture}[thm]{Conjecture}
\newtheorem{qn}[thm]{Question}
\newtheorem{pn}[thm]{Problem}

\numberwithin{equation}{section}

\theoremstyle{definition}
\newtheorem{definition}[thm]{Definition}
\newtheorem{remark}[thm]{Remark}
\newtheorem{ex}[thm]{Example}

\DeclareMathOperator{\im}{im} % image of a morphism
\DeclareMathOperator{\spec}{Spec} % Spec of a ring
\DeclareMathOperator{\proj}{Proj} % Proj of a graded ring
\newcommand{\sproj}{\proj} % Proj of a graded ring sheaf
\newcommand{\supp}{\operatorname{Supp}}
\newcommand{\iso}{\cong}
\newcommand{\niso}{\ncong}

\newcommand{\farg}{-} % argument of a functor
\newcommand{\id}{\mathrm{id}}
\newcommand{\dual}{^{\vee}} % dual
\newcommand{\comp}{\circ} % composition
\newcommand{\mor}[1]{\xrightarrow{#1}}
\newcommand{\lto}{\longrightarrow}
\newcommand{\mono}{\hookrightarrow} % monomorphism
\newcommand{\isomor}{\mor{\sim}} % isomorphism
\newcommand{\rest}[1]{|_{#1}} % restriction of function
\newcommand{\K}{\Bbbk} % field K
\newcommand{\cat}[1]{{\mathbf{#1}}} % category
\newcommand{\opp}{^{\circ}} % opposed category or functor
\newcommand{\s}[1]{\mathcal{#1}} % sheaf (over a topological space)
\newcommand{\so}{\s{O}} % structure sheaf O
\newcommand{\sod}{\s{O}_{\Delta}} % sheaf of the diagonal
\newcommand{\Hom}{\mathrm{Hom}}
\newcommand{\Ext}{\mathrm{Ext}}

\newcommand{\Aut}{\mathrm{Aut}}
\newcommand{\Pic}{\mathrm{Pic}} % Picard group
\newcommand{\cone}[1]{\mathsf{Cone}(#1)} % cone of a morphism
\newcommand{\rort}[1]{#1^{\perp}} % right orthogonal
\newcommand{\lort}[1]{{}^{\perp}#1} % left orthogonal
\newcommand{\sh}[2][1]{#2[#1]} % shift functor
\newcommand{\dg}{\mathrm{dg}}
\newcommand{\FM}[2][]{\Phi^{#1}_{#2}} % Fourier-Mukai functor
\newcommand{\aFM}[2][]{\Psi^{#1}_{#2}} % abelian Fourier-Mukai functor
\newcommand{\FMH}[1]{\FM[H]{#1}}
\newcommand{\FMK}[1]{\FM[K]{#1}}
\newcommand{\FMdg}[1]{\FM[\dg]{#1}} % Fourier-Mukai dg
\newcommand{\FMS}[1]{\FM[\mathrm{s}]{#1}}
\newcommand{\D}[1][]{\mathrm{D}^{#1}} % derived category
\newcommand{\Db}{\D[b]} % bounded derived category
\newcommand{\Dp}[1][]{\cat{Perf}_{#1}} % category of perfect complexes
\newcommand{\Dg}{\D[\dg]} % dg derived category
\newcommand{\Perf}{\mathrm{Perf}^{\,\dg}} % dg perfect category
\newcommand{\R}{\mathbf{R}} % right derived
\newcommand{\Lotimes}{\overset{\mathbf{L}}{\otimes}} % der. fun. of \otimes
\newcommand{\ds}{\omega} % dualizing sheaf
\newcommand{\fun}[1]{\mathsf{#1}} % functor
\newcommand{\dgfun}[1]{\fun{#1}^{\dg}} % dg-functor
\newcommand{\gfMod}[1]{\cat{gmod}\text{-}#1} % cat. of f.g. graded modules
\newcommand{\dgMod}[1]{\cal{M}od\text{-}#1} % category of dg-modules
\newcommand{\Coh}{\cat{Coh}}
\newcommand{\Qcoh}{\cat{Qcoh}}
\newcommand{\p}{\mathrm{p}} % (closed) point
\newcommand{\ort}[1]{\langle#1\rangle} % (semi)orthogonal decomposition
\newcommand{\ExFun}{\cat{ExFun}} % cat. of exact functors (bet. 2 triang. cat.)
\newcommand{\HH}{\mathrm{H\!H}}
\newcommand{\HO}{\mathrm{H}\Omega}
\newcommand{\HT}{\mathrm{HT}}
\newcommand{\kercat}[2]{\cat{K}(#1,#2)} % cat. of (abelian) kernels

\newcommand{\Cc}{(\ast)}

\newcommand{\ob}{\mathrm{Ob}}
\newcommand{\colim}{\underset{\longrightarrow n}{\mathrm{colim}}\,}
\newcommand{\cal}{\mathcal}
\newcommand{\ka}{{\cal A}}
\newcommand{\kb}{{\cal B}}
\newcommand{\kc}{{\cal C}}

\newcommand{\ke}{{\cal E}}
\newcommand{\kf}{{\cal F}}
\newcommand{\kg}{{\cal G}}

\newcommand{\ki}{{\cal I}}

\newcommand{\ko}{{\cal O}}
\newcommand{\kp}{{\cal P}}

\newcommand{\ks}{{\cal S}}
\newcommand{\kt}{{\cal T}}

\newcommand{\ZZ}{\mathbb{Z}}
\newcommand{\QQ}{\mathbb{Q}}
\newcommand{\RR}{\mathbb{R}}
\newcommand{\CC}{\mathbb{C}}

\newcommand{\PP}{\mathbb{P}}
\newcommand{\sHom}{\cal{H}om}

\begin{document}

	\title[Fourier--Mukai functors: a survey]{Fourier--Mukai functors: a survey}

	\author{Alberto Canonaco and Paolo Stellari}

	\address{A.C.: Dipartimento di Matematica ``F. Casorati'', Universit{\`a}
	degli Studi di Pavia, Via Ferrata 1, 27100 Pavia, Italy}
	\email{alberto.canonaco@unipv.it}

	\address{P.S.: Dipartimento di Matematica ``F.
	Enriques'', Universit{\`a} degli Studi di Milano, Via Cesare Saldini
	50, 20133 Milano, Italy}
	\email{paolo.stellari@unimi.it}
    \urladdr{http://users.mat.unimi.it/users/stellari/}
	
	\thanks{P.S.\ was partially supported by the MIUR of the Italian Government in the framework of the National Research Project ``Geometria algebrica e aritmetica, teorie coomologiche e teoria dei motivi'' (PRIN 2008).}

	\keywords{Derived categories, Fourier--Mukai functors}

	\subjclass[2010]{14F05, 18E10, 18E30}
	
		\begin{abstract}
			This paper surveys some recent results about Fourier--Mukai functors. In particular, given an exact functor between the bounded derived categories of coherent sheaves on two smooth projective varieties, we deal with the question whether this functor is of Fourier--Mukai type. Several related questions are answered and many open problems are stated.
		\end{abstract}

		\maketitle

	\section{Introduction}\label{Intro}
	
	Fourier--Mukai functors are ubiquitous in geometric contexts and the general belief is that they actually are \emph{the} geometric functors. Essentially, all known exact functors are of Fourier--Mukai type in the setting of proper schemes. This paper may be seen as an attempt to survey some recent works addressing this expectation according to several points of view.
	
	Let us first recall the definition of this kind of functors. Assume that $X_1$ and $X_2$ are smooth projective varieties over a field $\K$ and denote by $\Db(X_i):=\Db(\Coh(X_i))$ the bounded derived category of coherent sheaves on $X_i$. Given $\ke\in\Db(X_1\times X_2)$ we define the exact functor $\FM{\ke}\colon\Db(X_1)\to\Db(X_2)$ as
\begin{equation}\label{eqn:FM}
\FM{\ke}(-):=\R(p_2)_*(\ke\Lotimes p_1^*(-)),
\end{equation}
where $p_i\colon X_1\times X_2\to X_i$ is the natural projection. An exact functor $\fun{F}\colon\Db(X_1)\to\Db(X_2)$ is a \emph{Fourier--Mukai functor} (or of \emph{Fourier--Mukai type}) if there exists $\ke\in\Db(X_1\times X_2)$ and an isomorphism of exact functors $\fun{F}\iso\FM{\ke}$. The complex $\ke$ is called a \emph{kernel} of $\fun{F}$. This definition will be extended to more general settings in the course of the paper allowing $X_i$ to be singular or considering supported derived categories.

\smallskip

One of the first examples of these functors appeared in Mukai's seminal paper \cite{Mu1} dating 1981. Mukai studied what he originally called a \emph{duality} between the bounded derived category $\Db(A)$ of an abelian variety (or a complex torus) $A$ and the one of its dual variety $\hat{A}$. Such a duality is nothing but an equivalence
	\[
	\fun{F}\colon\Db(A)\lto\Db(\hat{A})
	\]
realized as a Fourier--Mukai functor whose kernel is precisely the universal Picard sheaf $\kp\in\Coh(A\times\hat{A})$. In other words, the inverse of $\fun{F}$ sends a skyscraper sheaf $\ko_\p$ (here $\p$ is a closed point of $\hat{A}$) on $\hat{A}$ to the degree $0$ line bundle $L_\p\in\Pic^0(A)$ parametrized by $\p$.

This discussion motivates the appearance of the word `Mukai' in the name of these functors. On the other hand, Mukai himself clarified why they should be thought of as a sort of Fourier transforms. Indeed, the push forward along the projection is the analogue of the integration while the Fourier--Mukai kernel is the same as the kernel in a Fourier transform.

A more precise historical reconstruction of the origins of the notion of Fourier--Mukai functor should certainly point to the paper \cite{SKK} where the notion of \emph{Fourier--Sato transform} was introduced (see also Section 3.7 in \cite{KS}). This is probably one of the first attempts to `categorify' the Fourier transform.

\smallskip

There are several possible directions along which to study these functors. In this paper, we are interested in the very specific but important question already mentioned at the beginning:

\smallskip

{\centerline{\emph{Are all exact functors between the bounded derived categories}}}

{\centerline{\emph{of smooth projective varieties of Fourier--Mukai type?}}}

\smallskip

\noindent This is certainly one of the main open problems in the literature concerning the special geometric incarnation of the theory of derived categories. Our aim is to survey the more recent approaches to it and, at the same time, to analyze other related questions concerning, for example, the uniqueness of the Fourier--Mukai kernels. The relevance of the question above cannot be overestimated. Indeed, once we know that an exact functor is of Fourier--Mukai type and the base field is $\CC$, then we can study its action on various cohomology groups and deform it along with the varieties. In Section \ref{sec:motivations} we survey some of these issues.

The main problems we want to consider are listed in Section \ref{subsec:questions}. The breakthroughs in the theory are contained in \cite{Or1} and, more recently, in \cite{LO}, where new inputs from the theory of dg-categories are taken into account. Namely,

\begin{itemize}
	\item[(A)] {\bf Orlov \cite{Or1}:} If $\fun{F}\colon\Db(X_1)\to\Db(X_2)$ is a fully faithful functor and $X_1$, $X_2$ are smooth projective varieties, then there exists a unique (up to isomorphism) $\ke\in\Db(X_1\times X_2)$ and an isomorphism of exact functors $\fun{F}\iso\FM{\ke}$ (see Theorem \ref{thm:Orlov}).
	
	\smallskip
	
	\item[(B)] {\bf Lunts--Orlov \cite{LO}:} The same holds when $X_1$ and $X_2$ are projective schemes and we deal with the categories of perfect complexes on them (see Theorem \ref{thm:LO}).
\end{itemize}
These two results will provide the two leading references in this paper. They will be explained in Sections \ref{sec:problems} and \ref{sec:existence} and, at the same time, we will study to which extent we may expect that they can be extended and generalized. The examples that seem to be encouraging in this direction are roughly the following (more precise statements are given in the forthcoming sections):

\begin{itemize}
	\item[(a)] {\bf To\"en \cite{To}:} Quasi-functors between dg-enhancements of the categories of perfect complexes on projective schemes (see Theorem \ref{thm:Toen}).
	
	\smallskip
	
	\item[(b)] Exact functors between the abelian categories of coherent sheaves on smooth projective varieties (see Proposition \ref{prop:exab} and \cite{CS}).
\end{itemize}
In both cases, one proves that these functors are of Fourier--Mukai type (in an appropriate sense) and that the kernel is unique (up to isomorphism). We also point to \cite{BZFN} (and \cite{Pr}) for results extending those in \cite{To}.

The fact that an optimistic point of view about extending (A) and (B) in full generality may be too much is discussed in Section \ref{sec:solutions}.

\smallskip

During the exposition we will explain and list several open problems appearing naturally in many geometric contexts. They will be presented all along the paper and, in particular, in Section \ref{sec:openproblems}. Motivations are discussed in Section \ref{sec:motivations}. Sections \ref{sec:problems} and \ref{sec:existence} deal with the main results and techniques now available in the literature. Of course, we do not pretend to be exhaustive and complete in our presentation. For example, other overviews on the subject but from completely different perspectives are in \cite{AH,vB} (and, of course, in \cite{H}).

\smallskip

\noindent{\bf Notation.} In the paper, $\K$ is a field. Unless
otherwise stated, all schemes are assumed to be of finite type and
separated over $\K$; similarly, all additive (in particular,
triangulated) categories and all additive (in particular, exact)
functors will be assumed to be $\K$-linear. An additive category will
be called $\Hom$-finite if the $\K$-vector space $\Hom(A,B)$ is finite
dimensional for any two objects $A$ and $B$. If $\cat{A}$ is an abelian
(or more generally an exact) category, $\D(\cat{A})$ denotes the
derived category of $\cat{A}$ and $\Db(\cat{A})$ its full subcategory
of bounded complexes. Unless stated otherwise, all functors are derived
even if, for simplicity, we use the same symbol for a functor and its
derived version.

	\section{Motivations}\label{sec:motivations}
	
	In this section we would like to motivate the relevance of Fourier--Mukai functors a bit more. We stress their appearance in moduli problems and we give indications concerning the way they induce actions on various cohomologies. The reader interested in an introduction about derived and triangulated categories in geometric contexts can have a look at \cite{H}.
	
\subsection{First properties and examples from moduli problems}\label{subsubsec:moduli}

There are several instances where Fourier--Mukai functors appear. To make this clear, we discuss some examples.

\begin{ex}\label{ex:FM}
	Let $X_1$ and $X_2$ be smooth projective varieties.
	
	(i) Given an object $\ke\in\Db(X_1)$, the functor $\fun{F}(-)=\ke\otimes(-)$ is of Fourier--Mukai type. Namely, its Fourier--Mukai kernel is the object $\Delta_*\ke$, where $\Delta\colon X_1\to X_1\times X_1$ is the diagonal embedding.
	
A special example is provided by the Serre functor of $X_i$ which is the exact equivalence $\fun{S}_{X_i}(-)=(-)\otimes\omega_{X_i}[\dim(X_i)]$, where $\omega_{X_i}$ is the dualizing sheaf of $X_i$. Hence $\fun{S}_{X_i}$ is of Fourier--Mukai type. For later use, set $S_{X_i}:=\omega_{X_i}[\dim(X_i)]$.
	
	(ii) For a given morphism $f\colon X_1\to X_2$, denote by $\Gamma_f$ its graph. Then $f_*$ is a Fourier--Mukai functor with kernel $\ko_{\Gamma_f}$. Analogously, one can show that $f^*$ is a Fourier--Mukai functor whose kernel is always $\ko_{\Gamma_f}$, providing now a functor $\Db(X_2)\to\Db(X_1)$.
\end{ex}

We list here a number of useful properties.

\begin{prop}\label{prop:FMpropert}
	Let $X_1$ and $X_2$ be smooth projective varieties over $\K$ and let $\FM{\ke}$ be a Fourier--Mukai functor.
	
	{\rm (i)} The left and right adjoints of $\FM{\ke}$ exist and are of Fourier--Mukai type with kernels $\ke_L:=\ke\dual\otimes p_2^*S_{X_2}$ and $\ke_R:=\ke\dual\otimes p_1^*S_{X_1}$ respectively, where $p_i\colon X_1\times X_2\to X_i$ is the projection.
	
	{\rm (ii)} The composition of two Fourier--Mukai functors is again of Fourier--Mukai type.
\end{prop}

\noindent We leave it to the reader to explicitly determine a kernel in (ii) above.

\medskip

Let us now see some more complicated but interesting examples. Indeed, soon after \cite{Mu1}, it was clear that Fourier--Mukai functors appear in many moduli problems. This is the case of K3 surfaces (i.e.\ smooth, compact, complex simply connected surfaces with trivial canonical bundle) and moduli spaces of stable sheaves on them. Following \cite{Mu2}, let $X$ be a projective K3 surface and $M$ a fine moduli space of stable sheaves on $X$ with topological invariants fixed in such a way that $M$ is again a projective K3 surface. The universal family $\ke\in\Coh(M\times X)$ associated to this moduli problem provides an equivalence of Fourier--Mukai type
\[
\FM{\ke}\colon\Db(M)\lto\Db(X)
\]
sending a skyscraper sheaf to a stable sheaf on $X$. Most remarkably, it was observed in \cite{Or1} that all K3 surfaces $Y$ such that $\Db(X)\iso\Db(Y)$ are actually isomorphic to moduli spaces of stable sheaves on $X$.

In higher dimensions the interplay between Fourier--Mukai functors, geometric problems and moduli interpretations of them have been extensively studied. There are many occurrences in the context of birational geometry and in the more modern theory of stability conditions due to Bridgeland. We refrain from discussing them in this paper.

\subsection{Action on (singular) cohomology}\label{subsubsec:cohomo}

Having a description of an exact functor as a Fourier--Mukai functor allows one to define an action on cohomologies and homologies of various types. This may be very useful to describe the groups of autoequivalences of the derived categories of smooth projective varieties, which are rather complicated algebraic objects as soon as the variety has trivial canonical bundle.

The first highly non-trivial example we have in mind is the group of autoequivalences of the derived category of a projective K3 surface $X$. This group has a very complicated structure coming from the presence of the so called \emph{spherical objects} in $\Db(X)$ (i.e.\ objects whose endomorphism graded algebra is isomorphic to the cohomology of a $2$-sphere). The idea proposed in \cite{Or1} is to approach the analysis of $\Aut(\Db(X))$ by studying its action on singular cohomology.

\medskip

To spell this out clearly, we start with some general remarks. Assume
that $X_1$ and $X_2$ are smooth complex projective varieties and let
$\FM{\ke}\colon\Db(X_1)\to\Db(X_2)$ be a Fourier--Mukai functor with
kernel $\ke\in\Db(X_1\times X_2)$. Then the induced morphism at the
level of Grothendieck groups is given by the morphism
$\FMK{[\ke]}\colon K(X_1)\to K(X_2)$ defined by
\[
\FMK{[\ke]}(e):=(p_2)_*([\ke]\cdot p_1^*(e)),
\]
where $p_i\colon X_1\times X_2\to X_i$ is the natural projection.

Going further, for $\kg\in\Db(X_i)$, one can consider the \emph{Mukai vector}
\[
v([\kg]):=\mathrm{ch}(\kg)\cdot\sqrt{\mathrm{td}(X_i)}
\]
of $\kg$. When the context is clear, we write $v(\kg)$ instead of
$v([\kg])$. Now the morphism $\FMK{[\ke]}\colon K(X_1)\to K(X_2)$
gives rise to a map $\FMH{v([\ke])}\colon H^*(X_1,\QQ)\to H^*(X_2,\QQ)$ such that
\[
\FMH{v([\ke])}\colon b\longmapsto (p_2)_*(v([\ke])\cdot p_1^*(b)).
\]
The Grothendieck--Riemann--Roch Theorem shows that the following
diagram commutes:
\begin{eqnarray}\label{eqn:coh}
\xymatrix{K(X_1)\ar[d]_{v(-)}\ar[rr]^{\FMK{[\ke]}}& &
K(X_2)\ar[d]^{v(-)}\\H^*(X_1,\QQ)\ar[rr]^{\FMH{v([\ke])}}& &
H^*(X_2,\QQ).}
\end{eqnarray}

From now on, given a Fourier--Mukai functor
$\FM{\ke}\colon\Db(X_1)\to\Db(X_2)$, we denote $\FMH{v([\ke])}$ by $\FMH{\ke}$. The following is a fairly easy remark from \cite{Or1}.

\begin{prop}\label{prop:FM1}
With the above assumptions, the morphism $\FMH{\ke}\colon H^*(X_1,\QQ)\to H^*(X_2,\QQ)$ is an isomorphism of $\QQ$-vector spaces if $\FM{\ke}$ is an equivalence.
\end{prop}

For a positive integer $n$, one may take the Hodge decomposition $H^n(X_i,\CC)\iso\bigoplus_{p+q=n}H^{p,q}(X_i)$. A Fourier--Mukai equivalence does not preserve such a decomposition as, in general, it does not preserve the grading of the cohomology rings. Nevertheless, one has the following.

\begin{prop}\label{prop:FM2}
If $\FM{\ke}$ is an equivalence, the morphism $\FMH{\ke}$ induces isomorphisms
\[
\bigoplus_{p-q=i}H^{p,q}(X_1)\iso\bigoplus_{p-q=i}H^{p,q}(X_2)
\]
for all integers $i$.
\end{prop}

The vector space $H^*(X_i,\CC)$ can be endowed with some additional structure. Namely, for $v=\sum v_j\in\bigoplus_j H^{j}(X_i,\CC)$, set $v\dual:=\sum\sqrt{-1}^jv_j$. Then, for all $v,w\in H^*(X_i,\CC)$, we can define the \emph{Mukai pairing}
\[
\langle v,w\rangle_{X_i}:=\int_{X_i}\mathrm{exp}(\mathrm{c}_1(X_i)/2).(v\dual. w).
\]

\begin{prop}\label{prop:FM3}
If $\FM{\ke}$ is an equivalence, then the morphism $\FMH{\ke}$ preserves the Mukai pairing.
\end{prop}

\smallskip

Before going back to specific examples, let us mention a property that will be discussed later on in a different context. Here we assume that $\FM{\ke},\FM{\kf}\colon\Db(X_1)\to\Db(X_2)$ are Fourier--Mukai functors and not necessarily equivalences.

\begin{lem}\label{lem:uniqcoh}
	If $\FMH{\ke}=\FMH{\kf}$, then $v([\ke])=v([\kf])$.
\end{lem}

\begin{proof}
	The morphisms $\FMH{\ke}$ and $\FMH{\kf}$ are induced by objects in $H^*(X_1\times X_2,\QQ)$. Now apply the K\"unneth decomposition for the cohomology of the product to get $v(\ke)=v(\kf)$.
\end{proof}

In particular, this means that the `cohomological Fourier--Mukai kernel' of cohomological Fourier--Mukai functors is \emph{always} uniquely determined. Due to what we will show in Section \ref{sec:solutions}, one can speak about \emph{the} action of a Fourier--Mukai functor, being independent of the choice of the Fourier--Mukai kernel.

\medskip

Assume now that $X_1$ and $X_2$ are projective K3 surfaces and take a Fourier--Mukai equivalence $\FM{\ke}\colon\Db(X_1)\to\Db(X_2)$. A remark by Mukai shows that $\FMH{\ke}$ induces an isomorphism of $\ZZ$-modules $H^*(X_1,\ZZ)\iso H^*(X_2,\ZZ)$ in this case. The total cohomology $H^*(X_i,\ZZ)$ endowed with the Mukai pairing and the Hodge structure mentioned in Proposition \ref{prop:FM2}, is called the \emph{Mukai lattice} and denoted by $\widetilde H(X_i,\ZZ)$. Using the action of equivalences on cohomology and a bit of lattice theory, one can prove the following.

\begin{prop}\label{prop:BrMa}{\bf (\cite{BM}, Proposition 5.3.)}
	Given a projective K3 surface $X$, the number of isomorphism classes of K3 surfaces $Y$ such that $\Db(X)\iso\Db(Y)$ is finite.
\end{prop}

Nevertheless such a number can be arbitrarily large.

\begin{prop}\label{prop:Os}{\bf (\cite{Og} and \cite{St})}
For any positive integer $N$, there exist non-isomorphic K3 surfaces $X_1,\ldots,X_N$ such that $\Db(X_i)\iso\Db(X_j)$ for $i,j=1,\ldots,N$.
\end{prop}

Two smooth projective varieties $X_1$ and $X_2$ such that $\Db(X_1)\iso\Db(X_2)$ are usually called \emph{Fourier--Mukai partners}. Notice that Proposition \ref{prop:BrMa} is a special instance of the following conjecture which is nothing but \cite[Conj.\ 1.5]{Ka1}.

\begin{conjecture}\label{conJ:finiteFMpartners}{\bf (Kawamata)}
	The number of Fourier--Mukai partners up to isomorphism of a smooth projective variety is finite.
\end{conjecture}

Abelian varieties satisfy this prediction as well (see \cite{Or}). In \cite{AT}, the authors provide further evidence for it.

\smallskip

To give one more important application of the discussion in this section, we can go back to the problem mentioned at the beginning of this section and use the structure of Fourier--Mukai functors to get a (partial) description of the group of autoequivalences of a K3 surface $X$. The following is the result of the papers \cite{HLOY,HMS,Or1}.

\begin{thm}\label{thm:HMS}
	For a K3 surface $X$, there exists a surjective morphism
	\[
	\Aut(\Db(X))\lto\mathrm{O}_+(\widetilde H(X,\ZZ))
	\]
	sending a Fourier--Mukai equivalence $\FM{\ke}$ to $\FMH{\ke}$.
\end{thm}

Here $\mathrm{O}_+(\widetilde H(X,\ZZ))$ is the group of Hodge isometries of the Mukai lattice preserving the orientation of some $4$-dimensional (real) vector subspace of $H^*(X,\RR)$.

\subsection{Hochschild homology, cohomology and deformations}\label{subsubsec:Hochschild}

For many geometric purposes, the cohomology theory one may want to
consider is Hochschild cohomology (and homology). More precisely, assume that a Fourier--Mukai equivalence $\FM{\ke}\colon\Db(X_1)\to\Db(X_2)$ between the bounded derived categories of the smooth complex projective varieties $X_1$ and $X_2$ is given. Then one may want to study (first order) deformations of $X_i$ compatible with deformations of the Fourier--Mukai kernel $\ke\in\Db(X_1\times X_2)$. To this end, we indeed have to study Hochschild cohomology and homology and the corresponding actions of $\FM{\ke}$.

\medskip

If $X$ is a smooth projective variety and $\omega_X$ is its dualizing
sheaf, we define $S_X$ as in Example \ref{ex:FM},
$S^{-1}_X:=\omega\dual_X[-\dim(X)]$ and $S^{\pm
  1}_{\Delta}:=(\Delta)_*S^{\pm 1}_X$, where $\Delta\colon X\mono X\times X$ is the diagonal embedding. The \emph{$i$-th Hochschild homology and cohomology groups}, $i\in\ZZ$, are respectively (see, for example, \cite{Cal1})
\[
\begin{split}
\HH_i(X)&:=\Hom_{\Db(X\times X)}(S^{-1}_{\Delta}[i],\ko_{\Delta})\iso\Hom_{\Db(X)}(\ko_X[i],\Delta^*\ko_{\Delta})\\
\HH^i(X)&:=\Hom_{\Db(X\times X)}(\ko_{\Delta},\ko_{\Delta}[i])\iso\Hom_{\Db(X)}(\Delta^*\ko_{\Delta},\ko_X[i]).
\end{split}
\]
Set $\HH_*(X):=\bigoplus_i\HH_i(X)$ and $\HH^*(X):=\bigoplus_i\HH^i(X)$. The \emph{Hochschild--Kostant--Rosenberg isomorphisms} are graded isomorphisms
\[
\begin{split}
I^X_\mathrm{HKR}\colon&\HH_*(X)\to\HO_*(X):=\bigoplus_i\HO_i(X)\\
I_X^\mathrm{HKR}\colon&\HH^*(X)\to\HT^*(X):=\bigoplus_i\HT^i(X),
\end{split}
\]
where $\HO_i(X):=\bigoplus_{q-p=i}H^p(X,\Omega_X^q)$ and $\HT^i(X):=\bigoplus_{p+q=i}H^p(X,\wedge^q\kt_X)$. One then defines the graded isomorphisms
\[
\begin{split}
&I^X_K=(\mathrm{td}(X)^{1/2}\wedge(-))\comp I^X_\mathrm{HKR}\\
&I_X^K=(\mathrm{td}(X)^{-1/2}\lrcorner(-))\comp I_X^\mathrm{HKR}.
\end{split}
\]

From \cite{Cal2,Cal1}, we get a functorial graded morphism $(\FM{\ke})_\HH\colon\HH_*(X_1)\to\HH_*(X_2)$. The following shows the compatibility between this action and the one described in Section \ref{subsubsec:cohomo}. It is based on \cite{Mark}.

\begin{thm}\label{thm:MSHoch}{\bf(\cite{MSHoch}, Theorem 1.2.)}
    Let $X_1$ and $X_2$ be smooth complex projective varieties and let $\ke\in\Db(X_1\times X_2)$.
    Then the following diagram
    \begin{equation*}
        \xymatrix{\HH_*(X_1)\ar[rr]^{(\FM{\ke})_\HH}\ar[d]_{I_K^{X_1}}&&\HH_*(X_2)\ar[d]^{I_K^{X_2}}\\ H^*(X_1,\CC)\ar[rr]^{\FMH{\ke}}&&H^*(X_2,\CC)}
    \end{equation*}
    commutes.
\end{thm}

If $\FM{\ke}$ is an equivalence, then
there exists also an action $(\FM{\ke})^\HH$ on Hochschild
cohomology induced by the functor
$\FM{\ke\boxtimes\kp}\colon\Db(X_1\times X_1)\to\Db(X_2\times X_2)$,
where $\kp\iso\ke_L\iso\ke_R$ is the kernel of the inverse of
$\FM{\ke}$, which sends $\ko_{\Delta_{X_1}}$ to $\ko_{\Delta_{X_2}}$
(see, for example, \cite[Remark 6.3]{H}).

Now the second Hochschild cohomolgy group controls first order deformations of a smooth projective variety. Hence, given a Fourier--Mukai equivalence $\FM{\ke}\colon\Db(X_1)\to\Db(X_2)$ and combining the actions $(\FM{\ke})_\HH$, $(\FM{\ke})^\HH$ and Theorem \ref{thm:MSHoch}, one can control first order deformations of $X_1$ and $X_2$ compatible with deformations of the Fourier--Mukai functor $\FM{\ke}$. This was done, for example, in \cite{HMS}.

Interesting recent developments are contained in \cite{ABP}, where the authors deal with fully faithful Fourier--Mukai functors whose kernel is a (shift of a) sheaf.

\section{The main problems and the first improvements}\label{sec:problems}

In this section we list the main problems that we want to address. The answers to them which are available in the literature will be presented in Section \ref{sec:solutions}. For the moment we content ourselves with a discussion of a celebrated result of Orlov about Fourier--Mukai functors. Various generalizations or attempts to weaken the hypotheses in this result are discussed in this section as well.

\subsection{The questions}\label{subsec:questions}

Assume for the moment that all the varieties are smooth and projective. The most important problems concerning Fourier--Mukai functors may be summarized by the following two questions:

\medskip
\begin{itemize}
\item[(1)] {\it Are all exact functors between the bounded derived categories of coherent sheaves on smooth projective varieties of Fourier--Mukai type?}
\medskip
\item[(2)] {\it Is the kernel of a Fourier--Mukai functor unique (up to isomorphism)?}
\end{itemize}
\medskip
A positive answer to the first one was conjectured in \cite{BLL} as a consequence of a conjecture about the possibility to lift all exact funtors to the corresponding dg-enhancements. In these terms, a positive or negative answer to the second one implies the uniqueness or non-uniqueness of such dg-lifts.

\medskip

We can now put these questions in a more general setting. Indeed, consider the category
$\ExFun(\Db(X_1),\Db(X_2))$ of exact functors
between $\Db(X_1)$ and $\Db(X_2)$ (with morphisms the natural
transformations compatible with shifts) and define the functor
\begin{equation}\label{eqn:fun}
\FM[X_1\to X_2]{\farg}\colon\Db(X_1\times X_2)\lto\ExFun(\Db(X_1),\Db(X_2))
\end{equation}
by sending $\ke\in\Db(X_1\times X_2)$ to the Fourier--Mukai functor $\FM{\ke}$. Thus we can formulate the following problems:
\medskip
\begin{itemize}
\item[(Q1)] {\it Is $\FM[X_1\to X_2]{\farg}$ essentially surjective?}
\medskip
\item[(Q2)] {\it Is $\FM[X_1\to X_2]{\farg}$ essentially injective?}
\medskip
\item[(Q3)] {\it Is $\FM[X_1\to X_2]{\farg}$ faithful?}
\medskip
\item[(Q4)] {\it Is $\FM[X_1\to X_2]{\farg}$ full?}
\medskip
\item[(Q5)] {\it Does $\ExFun(\Db(X_1),\Db(X_2))$ have a triangulated structure making $\FM[X_1\to X_2]{\farg}$ exact?}
\end{itemize}
\medskip
Clearly, (Q1) and (Q2) are precisely (1) and (2), respectively. C\u{a}ld\u{a}raru provided a negative answer to (Q3) in \cite[Example 6.5]{C} (see also \cite{Ploog}), while a negative answer to (Q5) was expected already in \cite{BLL, To}. Nevertheless, in the seminal paper \cite{Or1} a positive answer to (1) and (2) has been provided under some additional assumption on the exact functor. In the original formulation, it can be stated as follows:

\begin{thm}\label{thm:Orlov}{\bf (Orlov)}
	Let $X_1$ and $X_2$ be smooth projective varieties and let $\fun{F}\colon\Db(X_1)\to\Db(X_2)$ be an exact fully faithful functor admitting a left adjoint. Then there exists a unique (up to isomorphim) $\ke\in\Db(X_1\times X_2)$ such that $\fun{F}\iso\FM{\ke}$.	
\end{thm}

A generalization to smooth stacks (actually obtained as global quotients) is contained in \cite{Ka}. In the rest of this section and as a preparation for a complete discussion of (Q1)--(Q5) that will be carried out in Sections \ref{sec:solutions} and \ref{sec:existence}, we start discussing how one may try to weaken the hypotheses of the above result.

\subsection{Existence of adjoints}\label{subsubsec:adj}

Of course, in purely categorical terms, the existence of adjoints to a given functor is not automatic. In this section we will see a first approach, due to Bondal and Van den Bergh, to make this straightforward in the geometric setting we are dealing with.

\medskip

Let us start from the more general setting where $\cat{T}$ is an $\Ext$-finite triangulated category. This means that $\sum_n\dim_{\K}\Hom(A,B[n])<\infty$, for all $A,B\in\cat{T}$. Denote by $\cat{Vect}\text{-}\K$ the category of $\K$-vector spaces. A contravariant functor $\fun{H}\colon\cat{T}\to\cat{Vect}\text{-}\K$ is \emph{cohomological} if, given a distinguished triangle
\[
A\lto B\lto C
\]
in $\cat{T}$, the sequence
\[
\fun{H}(C)\lto\fun{H}(B)\lto\fun{H}(A)
\]
is exact in $\cat{Vect}\text{-}\K$. A cohomological contravariant functor $\fun{H}$ is \emph{of finite type} if $\dim_\K\bigoplus_i\fun{H}(A[i])<\infty$, for all $A\in\cat{T}$.

\begin{definition}\label{def:saturation}
	The triangulated category $\cat{T}$ is \emph{(right) saturated} if every cohomological contravariant functor $\fun{H}$ of finite type is representable, i.e.\ there exists $A\in\cat{T}$ and an isomorphism of functors
	\[
	\fun{H}\iso\Hom(-,A).
	\]
\end{definition}

\begin{remark}\label{rmk:nonsmooth}
(i) By the Yoneda Lemma, if a cohomological functor $\fun{H}$ is representable, then the object representing it is unique (up to isomorphism).  	 

(ii) In \cite{BB}, the authors provide examples of `geometric' categories which are not saturated. Namely, if $X$ is a smooth compact complex surface containing no compact curves, then $\Db(X)$ is not saturated. Examples in higher dimensions are given in \cite{Ogui}.
\end{remark}

In the smooth proper case one has the following result.

\begin{thm}\label{thm:BB}{\bf(\cite{BB}, Theorem 1.1.)}
	Assume that $X$ is a smooth proper scheme over $\K$. Then $\Db(X)$ is saturated.
\end{thm}

Now assume that $X_1$ and $X_2$ are smooth proper schemes. As an application of the above theorem, we get the following well-known result.

\begin{prop}\label{prop:exadj}
Any exact functor $\fun{F}\colon\Db(X_1)\to\Db(X_2)$ has left and right adjoints.
\end{prop}

\begin{proof}
For any $\kf\in\Db(X_2)$ the functor $\Hom(\fun{F}(-),\kf)$
	is representable by a unique $\ke\in\Db(X_1)$ due to Theorem \ref{thm:BB}. Setting
	$\fun{G}(\kf):=\ke$, by the Yoneda Lemma we get a functor $\fun{G}\colon\Db(X_2)\to\Db(X_1)$ which is right
	adjoint to $\fun{F}$. Since $\Db(X_1)$ and $\Db(X_2)$ have Serre
	functors, it is a very easy exercise to prove that $\fun{F}$ has also a left adjoint.
\end{proof}

Observe that, due to \cite[Prop.\ 1.4]{BK}, the right and left adjoints in the above statement are automatically exact.

\subsection{The algebricity assumption}\label{subsec:projectivity}

In this section we show in which sense it is important to work with algebraic varieties. In particular, we give examples of exact functors between the bounded derived categories of coherent sheaves on smooth compact complex manifolds which are not of Fourier--Mukai type.

\medskip

For this, let $X$ be a generic non-projective K3 surface. With this we mean a K3 surface $X$ such that $\Pic(X)=0$. The following surprising result shows that the abelian categories of coherent sheaves on those surfaces are not fine invariants (see, for example, \cite{MS} for a brief account about coherent sheaves and Chern characters in this setting).

\begin{thm}\label{thm:Verb}{\bf (\cite{V4})}
	Let $X_1$ and $X_2$ be generic non-projective K3 surfaces. Then there exists an equivalence of abelian categories $\Coh(X_1)\iso\Coh(X_2)$.
\end{thm}

\begin{remark}\label{rmk:Verb1}
	(i) In the case of smooth projective varieties $X_1$ and $X_2$ a result of Gabriel (see \cite[Cor.\ 5.24]{H} for an easy proof using Fourier--Mukai functors) asserts that exactly the converse holds. Namely $X_1\iso X_2$ if and only if $\Coh(X_1)\iso\Coh(X_2)$.
	
	(ii) The above result was proved in \cite{V5} for the case of generic non-projective complex tori as well.
\end{remark}

Now take two non-isomorphic generic non-projective K3 surfaces $X_1$ and $X_2$. Theorem \ref{thm:Verb} implies that there exists an exact equivalence
\[
\fun{F}\colon\Db(X_1)\lto\Db(X_2).
\]
One may then wonder whether all such equivalences are of Fourier--Mukai type.

\begin{prop}\label{prop:NoFM}
	Let $X_1$ and $X_2$ be non-isomorphic generic non-projective K3 surfaces and let $\fun{F}\colon\Db(X_1)\to\Db(X_2)$ be the exact equivalence induced by an exact equivalence $\Coh(X_1)\iso\Coh(X_2)$. Then $\fun{F}$ is not of Fourier--Mukai type.
\end{prop}

\begin{proof}
By assumption, $\fun{F}$ sends the minimal objects in $\Coh(X_1)$ to minimal objects in $\Coh(X_2)$ (recall that an object in an abelian category is minimal if it does not admit proper subobjects). In particular, following the same argument as in the proof of \cite[Cor.\ 5.24]{H}, we get that $\fun{F}$ sends skyscraper sheaves to skyscraper sheaves. Hence if $\fun{F}\iso\FM{\ke}$, for some $\ke\in\Db(X_1\times X_2)$, then there should be an isomorphism $f\colon X_1\to X_2$ and a line bundle $L\in\Pic(X_2)$ such that $\fun{F}\iso(L\otimes(-))\comp f_*$ (see, for example, \cite[Cor.\ 5.23]{H}). But this contradicts the assumption $X_1\not\iso X_2$.
\end{proof}

\subsection{Non fully faithful functors}\label{subsec:nonfullyfaithful}

Now we discuss how the fully faithfulness assumption can be removed. We first discuss a generalization of Theorem \ref{thm:Orlov} while later we observe that the faithfulness assumption is redundant anyway. Indeed full functors turn out to be automatically faithful.

\subsubsection{Negative Hom's and sheaves}\label{subsubsec:CS}

We now see a way to reduce the assumptions on the functor $\fun{F}$, that, to our knowledge, is the best one available in the literature in the context of smooth projective varieties. We will see later on how this has to be modified for perfect complexes on singular (projective) varieties. Some details about the key ingredients in the proof will be discussed in Section \ref{sec:solutions}.

\begin{thm}\label{thm:CS1} {\bf(\cite{CS}, Theorem 1.1.)}
Let $X_1$ and $X_2$ be smooth projective varieties and let $\fun{F}\colon\Db(X_1)\to\Db(X_2)$ be an
exact functor such that, for any $\kf,\kg\in\Coh(X_1)$,
\begin{eqnarray}\label{eqn:hyp}
\Hom_{\Db(X_2)}(\fun{F}(\kf),\fun{F}(\kg)[j])=0\;\;\mbox{if }j<0.
\end{eqnarray}
Then there exist $\ke\in\Db(X_1\times X_2)$
and an isomorphism of functors $\fun{F}\iso\FM{\ke}$. Moreover, $\ke$ is
uniquely determined up to isomorphism.
\end{thm}

A class of exact functors satisfying \eqref{eqn:hyp} is clearly provided by full functors. Unfortunately this is not a really interesting case, as in Section \ref{subsec:full} we will show that, in the present context, all full functors are actually automatically faithful.

\begin{ex}\label{ex:sathyp}
For a rather trivial example of a non-full exact functor satisfying
\eqref{eqn:hyp}, we can consider $\id\oplus\id\colon\Db(X)\to\Db(X)$,
where $X$ is a smooth projective variety. More generally, given a line
bundle $L\in\Pic(X)$, we can take $\FM{\Delta_*L}\oplus\FM{\Delta_*L}$
(see Example \ref{ex:FM}).
\end{ex}

\begin{ex}\label{ex:other}
Notice that all exact functors $\Db(X_1)\to\Db(X_2)$ obtained by
deriving an exact functor $\Coh(X_1)\to\Coh(X_2)$ are examples of
functors satisfying \eqref{eqn:hyp}.
\end{ex}

\begin{remark}\label{rmk:twisted}
	The original version of Theorem \ref{thm:CS1}, stated in \cite{CS}, deals with the more general notion of \emph{twisted variety} where condition \eqref{eqn:hyp} can be stated as well.
\end{remark}

\subsubsection{Full implies faithful}\label{subsec:full}

In this section we assume that $\K$ is algebraically closed of
characteristic $0$. Let $X_1$ and $X_2$ be smooth projective varieties
and assume that an exact functor $\fun{F}\colon\Db(X_1)\to\Db(X_2)$ is
full and such that $\fun{F}\not\iso 0$. By Theorem \ref{thm:CS1},
$\fun{F}$ is a Fourier--Mukai functor. So $\fun{F}\iso\FM{\ke}$, for
some $\ke\in\Db(X_1\times X_2)$.

There exists a very useful criterion to establish when a Fourier--Mukai functor $\FM{\ke}\colon\Db(X_1)\to\Db(X_2)$ is fully faithful.

\begin{thm}\label{thm:criterion}{\bf(\cite{BO} and \cite{Br})}
Under the assumptions above, $\FM{\ke}$ is fully faithful if and only if
	\[
	\Hom_{\Db(X_2)}(\FM{\ke}(\ko_{x_1}),\sh[i]{\FM{\ke}(\ko_{x_2})})\iso\begin{cases}\K & \text{if $x_1=x_2$ and $i=0$}\\0 & \text{if $x_1\neq x_2$ or $i\not\in[0,\dim(X_1)]$}\end{cases}
	\]
	for all closed points $x_1,x_2\in X_1.$
\end{thm}

Thus, because of this result and the
fact that $\fun{F}$ is full, to show that the functor is also faithful
it is enough to prove that there are no closed points $x\in X_1$ such
that $\Hom(\fun{F}(\ko_x),\fun{F}(\ko_x))=0$ or, in other words, such
that $\fun{F}(\ko_x)\iso 0.$

To see this, take the left adjoint $\fun{G}\colon\Db(X_2)\to\Db(X_1)$ of $\fun{F}$ and consider the
composition $\fun{G}\comp\fun{F}$ which is again a Fourier--Mukai
functor (see Proposition \ref{prop:FMpropert}), hence isomorphic to $\FM{\kf}$ for some $\kf\in\Db(X_1\times
X_1).$ Assume that there are $x_1,x_2\in X_1$ such that
$\fun{F}(\ko_{x_1})\not\iso 0$ while $\fun{F}(\ko_{x_2})\iso 0.$ By
\cite{BO} (see, in particular, Proposition 1.5 there) the Mukai vector $v(\FM{\kf}(\ko_{x_1}))$ is not zero.

On the other hand, by Propositions \ref{prop:FM1}, \ref{prop:FM2} and \ref{prop:FM3}, the functor $\FM{\kf}$
induces a morphism $\FMH{\kf}\colon H^*(X_1,\QQ)\to H^*(X_1,\QQ)$
such that
\[
0\neq v(\FM{\kf}(\ko_{x_1}))=\FMH{\kf}(v(\ko_{x_1}))=\FMH{\kf}(v(\ko_{x_2}))=v(\FM{\kf}(\ko_{x_2}))=0.
\]
This contradiction proves that, if $\fun{F}$ were not faithful, then
$\fun{F}(\ko_x)\iso 0$ for every closed point $x\in X$.

We claim that if this is true, then $\fun{F}\iso0$. Indeed let $\fun{G}$ and $\fun{H}$ be the left and right adjoints of $\fun{F}$. Of course, $\fun{G}\comp\fun{F}(\ko_x)\iso 0$, for all closed points $x$ in $X_1$. In particular, for all $n\in\ZZ$ and any $\kb\in\Db(X_1)$, we have
\[
0=\Hom(\fun{G}\comp\fun{F}(\ko_x),\kb[n])\iso\Hom(\ko_x,\fun{H}\comp\fun{F}(\kb)[n]).
\]
Therefore $\fun{H}\comp\fun{F}(\kb)\iso 0$, for all $\kb\in\Db(X_1)$. But now
\[
0=\Hom(\kb,\fun{H}\comp\fun{F}(\kb))\iso\Hom(\fun{F}(\kb),\fun{F}(\kb)).
\]
Thus we would get $\fun{F}(\kb)\iso 0$, for all $\kb\in\Db(X_1)$ and so we proved the following result.

\begin{thm}\label{thm:full}
	Let $X_1$ and $X_2$ be smooth projective varieties over an algebraically closed field of characteristic $0$ and assume that an exact functor $\fun{F}\colon\Db(X_1)\to\Db(X_2)$ is
	full. If $\fun{F}\not\iso 0$, then $\fun{F}$ is faithful as well.
\end{thm}

\begin{remark}\label{rmk:full}
	(i) Notice that in \cite{COS} a more general result is
  proved. In particular, the target category can be any triangulated
  category while the source category can be the category of perfect
  (supported) complexes on a noetherian scheme.
	
	(ii) One may easily extend the proof above to the case of twisted varieties. For this we just need to use the twisted version of the Chern character defined in \cite{HS} and again apply \cite[Prop.\ 1.5]{BO}. We leave this to the reader.
\end{remark}

\section{The (partial) answers to (Q2)--(Q5)}\label{sec:solutions}

We postpone for the moment the discussion about (Q1) which will be examined in Section \ref{sec:existence}. The remaining problems can be studied in a unitary way explained here below.

\subsection{Perfect complexes and good news}\label{subsec:goodnews}

We start our discussion with a case where all the above five questions
have a positive answer. In particular, this implies that (in the smooth
case) interesting examples answering these questions negatively have to
be searched for in dimension greater than zero.

\smallskip

We begin by extending the setting explained in the previous section. In
particular, let $X$ be a projective (not necessarily smooth) scheme
over $\K$. Denote by $\Dp(X)$ the category of \emph{perfect
complexes} on $X$ consisting of the objects in $\D(\Qcoh(X))$ which are quasi-isomorphic to bounded complexes of locally free sheaves of finite type over $X$. Obviously, $\Dp(X)\subseteq\Db(X)$ and the equality holds if and only if $X$ is regular.

The category $\Dp(X)$ coincides with the full subcategory of compact objects
in $\D(\Qcoh(X))$. Recall that an object $A$ in a triangulated
category $\cat{T}$ is \emph{compact} if, for each family of objects
$\{X_i\}_{i\in I}\subset\cat{T}$ such that $\bigoplus_i X_i$ exists in $\cat{T}$,
the canonical map
\[
\bigoplus_i\Hom(A,X_i)\lto\Hom\left(A,\oplus_i X_i\right)
\]
is an isomorphism.

In the singular setting we redefine the notion of
Fourier--Mukai functors once more since in general we cannot expect the Fourier--Mukai
kernels of exact functors $\Dp(X_1)\to\Dp(X_2)$ to be objects in
$\Dp(X_1\times X_2)$, but rather in $\Db(X_1\times X_2)$. More
precisely, one can show the following (see, for example,
\cite[Lemma 4.3]{CS1} for the proof).

\begin{lem}\label{lem:boundcoh}
Let $X_1$ and $X_2$ be projective schemes and let
$\ke\in\D(\Qcoh(X_1\times X_2))$ be an object such that
$\FM{\ke}\colon\D(\Qcoh(X_1))\to\D(\Qcoh(X_2))$ (defined as in
\eqref{eqn:FM}) sends $\Dp(X_1)$ to $\Db(X_2)$. Then
$\ke\in\Db(X_1\times X_2)$. Conversely, any $\ke\in\Db(X_1\times X_2)$
yields a Fourier--Mukai functor $\FM{\ke}\colon\Dp(X_1)\to\Db(X_2)$.
\end{lem}

Hence given two projective schemes $X_1$ and $X_2$ one can consider
the functor
\[
\FM[X_1\to X_2]{\farg}\colon\Db(X_1\times X_2)\lto\ExFun(\Dp(X_1),\Db(X_2))
\]
(which coincides with \eqref{eqn:fun} in the smooth case) and for it
one can again ask questions (Q1)--(Q5).

\smallskip

Now, if $X$ is a projective scheme over $\K$, it is an easy exercise
to show that every exact functor
$\fun{F}\colon\Dp(\spec\K)=\Db(\spec\K)\to\Db(X)$ is of Fourier--Mukai
type. More precisely, there exists an isomorphism of exact functors
$\fun{F}\iso\FM{\ke}$, where
\[
\ke:=\fun{F}(\ko_{\spec\K})\in\Db(X)=\Db(\spec\K\times X).
\]
It is also straightforward to see that the functor $\FM[\spec\K\to
X]{\farg}$ is an equivalence of categories, so that all the above
questions have a positive answer in this case.

\smallskip

If we exchange the role of $X$ and $\spec\K$ above, the situation
becomes slightly more complicated but nevertheless it is not difficult to see that
$\FM[X\to\spec\K]{\farg}$ is an equivalence as well. Indeed, as an easy consequence of \cite[Cor.\ 7.50]{R} (see also \cite[Thm.\ 3.3]{Ba1}), there is an equivalence
\[
\Db(X)\lto\ExFun(\Dp(X)\opp,\Db(\spec\K))
\]
and one can check that $\FM[X\to\spec\K]{\farg}$ is induced from this
by the exact anti-equivalence $\Dp(X)\iso\Dp(X)\opp$
sending $\kf$ to $\kf\dual$.

\subsection{Non-uniqueness of Fourier--Mukai kernels}\label{subsec:Q2}

The aim of this section is to prove that, even in the smooth case,
(Q2) has a negative answer in general. First observe that the functor $\FM[X_2\to X_1]{\farg}$ satisfies any of (Q1)--(Q5) if and only if $\FM[X_1\to X_2]{\farg}$ does. To see this, one identifies $\FM[X_2\to X_1]{\farg}$ with the opposite functor of $\FM[X_1\to X_2]{\farg}$ under the equivalences $\Db(X_1\times X_2)\to\Db(X_1\times X_2)\opp$
(defined on the objects by $\ke\mapsto\ke\dual\otimes
p_1^*\sh[d_1]{\ds_{X_1}}$) and
$\ExFun(\Db(X_1),\Db(X_2))\to\ExFun(\Db(X_2),\Db(X_1))\opp$ (defined
on the objects by $\fun{F}\mapsto\fun{F}_*$, the right adjoint of
$\fun{F}$). A key ingredient for this is Proposition
\ref{prop:exadj}. Here we set $d_i:=\dim(X_i)$.

\smallskip

For later use, we start studying the case of the projective line which provides a positive result related to (Q2).

\begin{lem}\label{lem:uniqPP1}
	If $X_1$ or $X_2$ is $\PP^1$, then $\FM[X_1\to X_2]{\farg}$ is
	essentially injective.
\end{lem}

	\begin{proof}
	As observed above, we can assume that
	$X_1=\PP^1$. Since on $\PP^1\times\PP^1$ there is a resolution of the
	diagonal of the form
	\[
	0\to\so_{\PP^1\times\PP^1}(-1,-1)\mor{x_0\boxtimes x_1-x_1\boxtimes x_0}\so_{\PP^1\times\PP^1}\to\sod\to0,
	\]
	the argument in \cite[Sect.\ 4.3]{CS} shows that, for every exact
	functor $\fun{F}\colon\Db(\PP^1)\to\Db(X_2)$, any object $\ke$ in
	$\Db(\PP^1\times X_2)$ such that $\fun{F}\iso\FM{\ke}$ is necessarily
	a convolution of the complex
	\[
	\so_{\PP^1}(-1)\boxtimes\fun{F}(\so_{\PP^1}(-1))
	\mor{\varphi:=x_0\boxtimes\fun{F}(x_1)-x_1\boxtimes\fun{F}(x_0)}
	\so_{\PP^1}\boxtimes\fun{F}(\so_{\PP^1}),
	\]
	hence it is uniquely determined up to isomorphism as the cone of
	$\varphi$.
\end{proof}

As soon as the genus of the curve grows, the situation becomes more complicated and, in a sense, more interesting. Indeed, we have the following result that is \cite[Thm.\ 1.1]{CS1}.

\begin{thm}\label{thm:nonuniq}
For every elliptic curve $X$ over an algebraically closed
field there exist $\ke_1,\ke_2\in\Db(X\times X)$ such that
$\ke_1\not\iso\ke_2$ but $\FM{\ke_1}\iso\FM{\ke_2}$.
\end{thm}

There is no space to explain the proof of this result in detail. Let us just mention how the two kernels are defined. By Serre duality,
\[
0\neq\Hom(\ko_\Delta,\ko_\Delta)\dual\iso\Hom(\ko_\Delta[-1],\ko_\Delta[1]),
\]
where $\ko_\Delta=\Delta_*\ko_X\in\Db(X\times X)$. For $0\neq\alpha\in\Hom(\ko_\Delta[-1],\ko_\Delta[1])$, we set
\[
\ke_1:=\ko_\Delta\oplus\ko_\Delta[1]\qquad\ke_2:=\cone{\alpha}.
\]

It makes then perfect sense to pose the following.

\begin{pn}\label{pn:anygenus}
Extend the non-uniqueness result in Theorem \ref{thm:nonuniq} to any
curve of genus $\ge1$.
\end{pn}

\smallskip

In \cite{CS1} we provided our best approximation to the uniqueness of the Fourier--Mukai kernels.

\begin{thm}\label{thm:cohm}{\bf (\cite{CS1}, Theorem 1.2.)}
Let $X_1$ and $X_2$ be projective schemes and let
$\fun{F}\colon\Dp(X_1)\to\Db(X_2)$ be an exact functor. If
$\fun{F}\iso\FM{\ke}$ for some $\ke\in\Db(X_1\times X_2)$, then the
cohomology sheaves of $\ke$ are uniquely determined (up to
isomorphism) by $\fun{F}$.
\end{thm}

Using the discussion in Section \ref{subsubsec:cohomo} we can derive the following straightforward consequence from the above result. We will always assume that $X_1$ and $X_2$ are smooth projective varieties.

\begin{cor}\label{cor:Kcohom}
Let $\ke_1,\ke_2\in\Db(X_1\times X_2)$ be such that $\FM{\ke_1}\iso\FM{\ke_2}\colon\Db(X_1)\to\Db(X_2)$.
Then $[\ke_1]=[\ke_2]$ in $K(X_1\times X_2)$ and so $\FMK{\ke_1}=\FMK{\ke_2}$ and $\FMH{\ke_1}=\FMH{\ke_2}$.
\end{cor}

\subsection{The remaining questions (Q3)--(Q5)}\label{subsec:Q}

Let us first consider the case of smooth projective curves.
	
\begin{prop}\label{nonfaith}{\bf (\cite{CS1}, Proposition 2.3.)}
Set $d_i:=\dim(X_i)$. If $\min\{d_1,d_2\}=1$, then $\FM[X_1\to
X_2]{\farg}$ is neither faithful nor full.
\end{prop}

\begin{proof}
We give a full proof only of the non-faithfulness, as it plays a role in the study of (Q5) below. As above, we can
assume that $1=d_1\le d_2$. Hence take a finite morphism $f\colon
X_1\to\PP^{d_2}$ and a finite and surjective (hence flat) morphism
$g\colon X_2\to\PP^{d_2}$. Then $\fun{F}:=g^*\comp
f_*\colon\Coh(X_1)\to\Coh(X_2)$ is an exact functor, which trivially
extends to an exact functor again denoted by
$\fun{F}\colon\Db(X_1)\to\Db(X_2)$. Clearly there exists
$0\niso\ke\in\Db(X_1\times X_2)$ such that $\fun{F}\iso\FM{\ke}$ (see Example \ref{ex:FM} and Proposition \ref{prop:FMpropert}).

Now observe that, by Serre duality,
\[
\Hom_{\Db(X_1\times X_2)}(\ke,\ke)\iso\Hom_{\Db(X_1\times X_2)}
(\ke,\ke\otimes\sh[1+d_2]{\ds_{X_1\times X_2}})\dual,
\]
so there exists $0\ne\alpha\in\Hom_{\Db(X_1\times X_2)}
(\ke,\ke\otimes\sh[1+d_2]{\ds_{X_1\times X_2}})$. Since
$\ds_{X_1\times X_2}\iso p_1^*\ds_{X_1}\otimes p_2^*\ds_{X_2}$, this
induces, for any $\s{F}\in\Coh(X_1)$, a morphism
\[
\FM{\alpha}(\s{F})\colon\FM{\ke}(\s{F})\iso\fun{F}(\s{F})\to
\FM{\ke\otimes\sh[1+d_2]{\ds_{X_1\times X_2}}}(\s{F})\iso
\fun{F}(\s{F}\otimes\ds_{X_1})\otimes\sh[1+d_2]{\ds_{X_2}}.
\]
As $\fun{F}(\s{F})$ and $\fun{F}(\s{F}\otimes\ds_{X_1})$ are objects
of $\Coh(X_2)$, it follows that $\FM{\alpha}(\s{F})=0$, whence
$\FM{\alpha}=0$ because every object of $\Db(X_1)$ is isomorphic to
the direct sum of its (shifted) cohomology sheaves (since the abelian
category $\Coh(X_1)$ is hereditary).

As for non-fullness, we prove it only when $X_1=X_2=X$ is an elliptic
curve and $\K$ is algebraically closed. By Theorem \ref{thm:nonuniq}
there exist $\ke_1,\ke_2\in\Db(X\times X)$ with $\ke_1\not\iso\ke_2$
and an isomorphism $\psi\colon\FM{\ke_1}\isomor\FM{\ke_2}$. Then we
claim that there is no morphism $f\colon\ke_1\to\ke_2$ such that
$\psi=\FM[X\to X]{f}$. Indeed, assume that such an $f$ exists. Then it
can be completed to a distinguished triangle
\[
\xymatrix{
\ke_1\ar[r]^-{f}&\ke_2\ar[r]&\kg,
}
\]
for some $\kg\in\Db(X_1\times X_2)$. By assumption $\FM{\kg}(\ka)=0$,
for all $\ka\in\Db(X_1)$. Therefore $\FM{\kg}\iso0$, whence $\kg\iso0$
by Theorem \ref{thm:CS1}. But then $f$ would be an isomorphism,
contradicting the assumption $\ke_1\not\iso\ke_2$.
	\end{proof}
	
We finally recall how (Q5) is studied in \cite{CS1}. For this we need
a couple of easy lemmas.

\begin{lem}\label{trivcone}
Let $\cat{T}$ be a $\Hom$-finite triangulated category and let $f\colon
A\to B$ be a morphism of $\cat{T}$. Then $\cone{f}\iso\sh{A}\oplus
B$ if and only if $f=0$.
\end{lem}

\begin{proof}
The other implication being well-known, we assume that
$\cone{f}\iso\sh{A}\oplus B$. Applying the cohomological functor
$\Hom(\farg,B)$ to the distinguished triangle $A\mor{f}B\to
\sh{A}\oplus B\to\sh{A}$, one gets an exact sequence of finite
dimensional $\K$-vector spaces
\[\Hom(\sh{A},B)\to\Hom(\sh{A}\oplus B,B)\to\Hom(B,B)
\mor{(\farg)\comp f}\Hom(A,B).\]
For dimension reasons, the last map must be $0$, hence $f=0$.
\end{proof}

\begin{lem}\label{injfaith}
Let $\fun{F}\colon\cat{T}\to\cat{T}'$ be an exact functor between
triangulated categories and assume that $\cat{T}$ is
$\Hom$-finite. If $\fun{F}$ is essentially injective, then $\fun{F}$
is faithful, too.
\end{lem}

\begin{proof}
Let $f\colon A\to B$ be a morphism of $\cat{T}$ such that
$\fun{F}(f)=0$. Then
\[\fun{F}(\cone{f})\iso\cone{\fun{F}(f)}\iso
\sh{\fun{F}(A)}\oplus\fun{F}(B)\iso\fun{F}(\sh{A}\oplus B)\] in
$\cat{T}'$, whence $\cone{f}\iso\sh{A}\oplus B$ in $\cat{T}$ because
$\fun{F}$ is essentially injective. It follows from Lemma
\ref{trivcone} that $f=0$.
\end{proof}

Recollecting the above results, we get the following.

\begin{prop}\label{prop:notria}{\bf (\cite{CS1}, Corollary 2.7.)}
If $d_1,d_2>0$ and $X_1$ or $X_2$ is $\PP^1$, then there is no
triangulated structure on $\ExFun(\Db(X_1),\Db(X_2))$
such that $\FM[X_1\to X_2]{\farg}$ is exact.
\end{prop}

\begin{proof}
This follows from Lemma \ref{injfaith}, since we know that in this case
$\FM[X_1\to X_2]{\farg}$ is essentially injective by Lemma
\ref{lem:uniqPP1}, but not faithful by Proposition \ref{nonfaith}.
\end{proof}

Notice that, as observed in \cite{To}, there is no natural triangulated structure on the category $\ExFun(\Db(X_1),\Db(X_2))$. One can then pose the following question.

\begin{pn}\label{pn:morecases}
Understand whether there may be smooth projective varieties $X_1$ and $X_2$ of
positive dimension such that (Q5) has a positive answer.
\end{pn}
	
\section{Existence of Fourier--Mukai kernels and (Q1)}\label{sec:existence}

We are now ready to discuss the partial answers to (Q1) actually present in the literature. As we have already observed, we need to impose rather strong conditions on the exact functors in order to get nice results.

\subsection{The non-smooth case}\label{subsec:nonsmooth}

The idea of studying Fourier--Mukai functors between triangulated categories associated to singular varieties explained in the baby examples in Section \ref{subsec:goodnews} has been extensively analyzed in \cite{LO} using new ideas coming from dg-categories. Let us start from the following result.

\begin{prop}\label{prop:LO1}{\bf(\cite{LO}, Corollary 9.12.)}
Let $X_1$ and $X_2$ be quasi-compact separated schemes over
$\K$. Assume that $X_1$ has enough locally free sheaves and let
$\fun{F}\colon\Dp(X_1)\to\D(\Qcoh(X_2))$ be a fully faithful exact
functor that commutes with direct sums. Then there is an
$\ke\in\D(\Qcoh(X_1\times X_2))$ such that the functor $\FM{\ke}$ is
fully faithful and
\begin{equation}\label{eqn:FMobj}
\FM{\ke}(\ka)\iso\fun{F}(\ka)
\end{equation}
for any $\ka\in\Dp(X_1)$.	
\end{prop}

Needless to say, the existence of the isomorphism \eqref{eqn:FMobj} is a rather weak condition because, already in the smooth case, it may not extend to an isomorphism of functors. To show that this is possible, consider the case of $\PP^1\times\PP^1$. Exactly as in Section \ref{subsec:Q2}, observe that, by Serre duality,
\[
0\neq\Hom(\ko_\Delta,\ko_\Delta)\dual\iso\Hom(\ko_\Delta[-1],\ko_\Delta\otimes\omega_{\PP^1\times\PP^1}[1]).
\]
Hence take a non-trivial $\alpha\colon\ko_\Delta[-1]\to\ko_\Delta\otimes\omega_{\PP^1\times\PP^1}[1]\iso\Delta_*\omega_{\PP^1}^{\otimes 2}[1]$ and consider the objects
\[
\ke_1:=\ko_\Delta\oplus\Delta_*\omega_{\PP^1}^{\otimes 2}[1]\qquad\ke_2:=\cone{\alpha}.
\]
Then one has the following easy result.

\begin{lem}\label{lem:isoobj}
	For every $\ka\in\Db(\PP^1)$ we have $\FM{\ke_1}(\ka)\iso\FM{\ke_2}(\ka)$ but $\FM{\ke_1}\not\iso\FM{\ke_2}$.
\end{lem}

\begin{proof}
The existence of an isomorphism $\FM{\ke_1}(\ka)\iso\FM{\ke_2}(\ka)$ for any $\ka\in\Db(\PP^1)$ is obvious. The fact that $\FM{\ke_1}\not\iso\FM{\ke_2}$ follows from the uniqueness of Fourier--Mukai kernels for $\PP^1$ (see Lemma \ref{lem:uniqPP1}) and the fact that $\ke_1\not\iso\ke_2$.
\end{proof}

On the other hand, putting some more hypotheses on the schemes, we get a global isomorphism, as stated in the following theorem which is \cite[Cor.\ 9.13]{LO}. For a scheme $X$, denote by $T_0(\ko_X)$ the maximal $0$-dimensional torsion subsheaf of $\ko_{X}$.

\begin{thm}\label{thm:LO}{\bf(Lunts--Orlov)}
Let $X_1$ be a projective scheme with $T_0(\ko_{X_1})=0$ and
assume that $X_2$ is a noetherian separated scheme over $\K$. Given an exact fully
faithful functor $\fun{F}\colon\Dp(X_1)\to\Db(X_2)$,
there are an $\ke\in\Db(X_1\times X_2)$ and an isomorphism of exact functors
$\FM{\ke}\iso\fun{F}$.
\end{thm}

\begin{remark}\label{rmk:LO}
The kernel turns out to be unique in perfect analogy with Theorem
\ref{thm:Orlov}. This is observed in \cite{CS2}, following a
suggestion by Orlov.
\end{remark}

There is another approach to the Fourier--Mukai functors in the
non-smooth case due to Ballard.

\begin{thm}\label{thm:Ballard}{\bf(\cite{Ba1}, Theorem 1.2.)}
Let $X_1$ and $X_2$ be projective schemes such that $T_0(\ko_{X_1})=0$. If
$\fun{F}\colon\Dp(X_1)\to\Dp(X_2)$ is a fully faithful exact functor
with left and right adjoints, then there are an $\ke\in\Db(X_1\times
X_2)$ and an isomorphism of exact functors $\FM{\ke}\iso\fun{F}$.
\end{thm}

As remarked in \cite{Ba1}, contrary to the smooth case, the existence of the adjoints is not automatic at all. On the other hand, the proof of Theorem \ref{thm:Ballard} differs from the one of Theorem \ref{thm:LO} as it does not make use of dg-categories and is closer to the spirit of the one of Theorem \ref{thm:Orlov}.

\subsection{Some ingredients in the proof of Theorem \ref{thm:LO}}\label{subsec:ingredients}

A complete account of the details of the proof of Theorem \ref{thm:LO} is far beyond the scope of this paper. Nevertheless, there are at least three main steps in it which we want to highlight as they provide sources of interesting (and difficult) open problems.

\subsubsection{Dg-categories}\label{subsubsec:dg}

First one wants to find an object $\ke\in\Db(X_1\times X_2)$ to compare the functors $\fun{F}$ and $\FM{\ke}$. This is done by passing to dg-ehnacements and using a celebrated result of To\"{e}n.

\smallskip

Recall that a \emph{dg-category} is an additive category $\cat{A}$ such that, for all $A,B\in\ob(\cat{A})$, the morphism spaces $\Hom(A,B)$ are $\ZZ$-graded $\K$-modules with a differential $d\colon\Hom(A,B)\to\Hom(A,B)$ of degree $1$ compatible with the composition.

Given a dg-category $\cat{A}$ we denote by $H^0(\cat{A})$ its
\emph{homotopy} category. The objects of $H^0(\cat{A})$ are the same
as those of $\cat{A}$ while the morphisms are obtained by taking the
$0$-th cohomology $H^0(\Hom_{\cat{A}}(A,B))$ of the complex
$\Hom_{\cat{A}}(A,B)$. If $\cat{A}$ is pre-triangulated (see \cite{K}
for the definition), then $H^0(\cat{A})$ has a natural structure of
triangulated category.

A \emph{dg-functor} $\fun{F}\colon\cat{A}\to\cat{B}$ is the datum of a map $\mathrm{Ob}(\cat{A})\to\mathrm{Ob}(\cat{B})$ and of morphisms
of dg $\K$-modules $\Hom_{\cat{A}}(A,B)\to\Hom_{\cat{B}}(\fun{F}(A),\fun{F}(B))$, for $A,B\in\mathrm{Ob}(\cat{A})$, which are compatible with the composition and the units.

For a small dg-category $\cat{A}$, one can consider the
pre-triangulated dg-category $\dgMod{\cat{A}}$ of \emph{right dg
  $\cat{A}$-modules}. A right dg $\cat{A}$-module is a dg-functor
$\fun{M}\colon\cat{A}\opp\to\dgMod{\K}$, where $\dgMod{\K}$ is the
dg-category of dg $\K$-modules. The full dg-subcategory of acyclic
right dg-modules is denoted by $\ka c(\cat{A})$, and $H^0(\ka
c(\cat{A}))$ is a full triangulated subcategory of the homotopy
category $H^0(\dgMod{\cat{A}})$. Hence the \emph{derived category} of
the dg-category $\cat{A}$ is the Verdier quotient
\[
\Dg(\cat{A}):=H^0(\dgMod{\cat{A}})/H^0(\ka c(\cat{A})).
\]

According to \cite{K,To}, given two dg-categories $\cat{A}$ and $\cat{B}$, we denote by $\mathrm{rep}(\cat{A},\cat{B})$ the full subcategory of the derived category
$\Dg(\cat{A}\opp\otimes\cat{B})$ of $\cat{A}$-$\cat{B}$-bimodules $\fun{C}$ such that the functor
$(-)\otimes_{\cat{A}}\fun{C}\colon\Dg(\cat{A})\to\Dg(\cat{B})$  sends the representable $\cat{A}$-modules to objects which are isomorphic to representable $\cat{B}$-modules.
A \emph{quasi-functor} is an object in $\mathrm{rep}(\cat{A},\cat{B})$
which is represented by a dg-functor $\cat{A}\to\dgMod{\cat{B}}$ whose
essential image consists of dg $\cat{B}$-modules quasi-isomorphic to
representable $\cat{B}$-modules. Notice that a quasi-functor
$\fun{M}\in\mathrm{rep}(\cat{A},\cat{B})$ defines a functor
$H^0(\fun{M})\colon H^0(\cat{A})\to H^0(\cat{B})$.

Given two pre-triangulated dg-categories $\cat{A}$ and $\cat{B}$, a \emph{dg-lift} of an exact functor $\fun{F}\colon H^0(\cat{A})\to H^0(\cat{B})$ is a quasi-functor $\fun{G}\in\mathrm{rep}(\cat{A},\cat{B})$ such that $H^0(\fun{G})\iso\fun{F}$.

An \emph{enhancement} of a triangulated category $\cat{T}$ is a pair
$(\cat{A},\alpha)$, where $\cat{A}$ is a pre-triangulated
dg-category and $\alpha\colon H^0(\cat{A})\to\cat{T}$ is an
exact equivalence. The enhancement $(\cat{A},\alpha)$ of $\cat{T}$ is
\emph{unique} if for any enhancement $(\cat{B},\beta)$ of $\cat{T}$
there exists a quasi-functor $\gamma\colon\cat{A}\to\cat{B}$ such that
$H^0(\gamma)\colon H^0(\cat{A})\to H^0(\cat{B})$ is an exact equivalence.

\begin{ex}\label{ex:dgcat}
	For $X$ a quasi-compact quasi-separated scheme, let $\mathrm{C}^{\dg}(X)$ be the dg-category of unbounded complexes of objects in $\Qcoh(X)$. Denote by $\mathrm{Ac}^{\dg}(X)$ the full dg-subcategory of $\mathrm{C}^{\dg}(X)$ consisting of acyclic complexes. Following \cite{Dr}, we take the quotient $\Dg(X):=\mathrm{C}^{\dg}(X)/\mathrm{Ac}^{\dg}(X)$ which is again a dg-category. This dg-category $\Dg(X)$ is pre-triangulated and $H^0(\Dg(X))\iso\D(\Qcoh(X))$ (see \cite{K,To}). Therefore it is an enhancement of $\D(\Qcoh(X))$.
	
	Consider then the full dg-subcategory $\Perf(X)$ whose objects are all the perfect complexes in $\D(\Qcoh(X))$. It turns out (see, for example, \cite[Sect.\ 1]{LO}) that $\Perf(X)$ is an enhancement of $\Dp(X)$.
\end{ex}

The following result answers positively a conjecture in \cite{BLL}. The reader can have a look at \cite[Sect.\ 9]{LO} for stronger statements.

\begin{thm}\label{thm:uniqen}{\bf(\cite{LO}, Theorem 7.9.)}
	The triangulated category $\Dp(X)$ on a quasi-projective scheme $X$ has a unique enhancement.
\end{thm}

\smallskip

Given a functor $\fun{F}\colon\Dp(X_1)\to\Db_{\Coh}(\Qcoh(X_2))$ as in the statement of Theorem \ref{thm:LO}, Lunts and Orlov construct in a highly non-trivial way a quasi-functor $\dgfun{F}\colon\Perf(X_1)\to\Dg(X_2)$. Now one can use the following.

\begin{thm}\label{thm:Toen}{\bf(\cite{To}, Theorem 8.9.)}
	Let $X_1$ and $X_2$ be quasi-compact and separated schemes over $\K$. Then we have a canonical quasi-equivalence
	\[
	\Dg(X_1\times X_2)\isomor\R\sHom_c(\Dg(X_1),\Dg(X_2)),
	\]
	where $\R\sHom_c$ denotes the dg-category formed by the direct sums preserving quasi-functors (i.e.\ their homotopy functors do).	
\end{thm}

Hence there are an $\ke\in\Dg(X_1\times X_2)$ and an isomorphism
$\dgfun{F}\iso\FMdg{\ke}$ and it remains to show that
$\fun{F}\iso H^0(\dgfun{F})\iso\FM{\ke}$.

\subsubsection{Ample sequences}\label{subsubsec:ample}

The projectivity assumption in the statement has a rather important role. Indeed one needs to work with ample sequences according to the following.

\begin{definition}\label{def:ample}
Given a $\Hom$-finite abelian category
$\cat{A}$, a subset
$\{P_i\}_{i\in\ZZ}\subset\ob(\cat{A})$ is an \emph{ample sequence}
if, for any $B\in\ob(\cat{A})$, there exists an integer $i(B)$ such
that, for any $i\leq i(B)$,
\begin{enumerate}
\item\label{ample1} the natural morphism $\Hom_{\cat{A}}(P_i,B)
\otimes P_i\to B$ is surjective;

\item\label{ample2} if $j\ne0$ then
$\Hom_{\Db(\cat{A})}(P_i,B[j])=0$;

\item\label{ample3} $\Hom_{\cat{A}}(B,P_i)=0$.
\end{enumerate}\end{definition}

If $X$ is a projective scheme and $H$ is an ample line bundle on $X$,
then one may consider the set $\cat{C}$ (often identified with the
corresponding full subcategory of $\Coh(X)$) consisting of objects of
the form $\ko_X(iH)$, where $i$ is any integer.

\begin{prop}\label{prop:ampleornot}{\bf(\cite{LO}, Proposition 9.2.)}
	If $X$ is a projective scheme such that $T_0(\ko_X)=0$, then $\cat{C}$ forms an ample sequence in the abelian category $\Coh(X)$.
\end{prop}

Notice that this is the place where the assumption about the maximal torsion subsheaf plays a distinguished role. Thus there is space for further work:

\begin{pn}\label{prob:maxtors}
	Remove the assumption $T_0(\ko_X)=0$ and, in particular, find a way to extend Theorem \ref{thm:LO} when $X_1$ is a $0$-dimensional projective scheme.
\end{pn}

At this point Lunts and Orlov show that the Fourier--Mukai functor $\FM{\ke}$, with kernel found in Section \ref{subsubsec:dg}, and the given functor $\fun{F}$ are such that there is an isomorphism
\begin{equation}\label{eqn:iso1}
	\theta_1\colon\fun{F}\rest{\cat{C}}\isomor\FM{\ke}\rest{\cat{C}}.
\end{equation}

Before discussing how this isomorphism can be extended, let us formulate the following rather general problem.

\begin{pn}\label{prob:ample}
	Avoid the use of ample sequences and relax the projectivity assumptions.
\end{pn}

Both Problem \ref{prob:maxtors} and \ref{prob:ample} are widely open but we believe that any improvement in these directions may give new important impulses to the theory.

\subsubsection{Convolutions}\label{subsubsec:convolutions}

The extension of \eqref{eqn:iso1} is achieved in two steps. First the extension takes place on the level of sheaves. And for this one writes every perfect sheaf (i.e.\ a coherent sheaf which is a perfect object as well) as a convolution of objects in the ample sequence $\cat{C}$ on $X_1$ described in the previous section.

\smallskip

Following \cite{Ka,Or1}, recall that a bounded complex in a
triangulated category $\cat{T}$ is a sequence of objects and
morphisms in $\cat{T}$
\begin{equation}\label{eqn:complex}
A_{m}\mor{d_{m}}A_{m-1}\mor{d_{m-1}}\cdots\mor{d_{1}}A_0
\end{equation}
such that $d_{j}\comp d_{j+1}=0$ for $0<j<m$. A \emph{right convolution} of
\eqref{eqn:complex} is an object $A$ together with a morphism
$d_0\colon A_0\to A$ such that there exists a diagram in $\cat{T}$
\[\xymatrix{A_m \ar[rr]^{d_m} \ar[dr]_{\id} \ar@{}[drr]|{\circlearrowright}
& & A_{m-1} \ar[rr]^{d_{m-1}} \ar[dr] \ar@{}[drr]|{\circlearrowright} & &
\cdots \ar[rr]^{d_2} & & A_1 \ar[rr]^{d_1} \ar[dr]
\ar@{}[drr]|{\circlearrowright} & & A_0 \ar[dr]_{d_0} \\
 & A_m \ar[ru] & & C_{m-1} \ar[ll]^{[1]} \ar[ru] & & \cdots
\ar[ll]^{[1]} & & C_1 \ar[ll]^{[1]} \ar[ru] & & A, \ar[ll]^{[1]}}\]
where the triangles with a $\circlearrowright$ are
commutative and the others are distinguished.

\smallskip

Roughly speaking, in this part of the argument, we have
$A\in\Coh(X_1)\cap\Dp(X_1)$ while $A_i$ is a finite direct sum of
objects in $\cat{C}$, for all $i$. Unfortunately, to use convolutions
one needs to make assumptions on the functor $\fun{F}$. The hypothesis
in Theorem \ref{thm:LO} that $\fun{F}$ is fully faithful goes
exactly in this direction. Thus, if we want to substantially improve
Theorem \ref{thm:LO}, one has to address the following:

\begin{pn}\label{prob:conv}
	Avoid the use of convolutions.
\end{pn}

All in all, we get an isomorphism
\[
	\theta_2\colon\fun{F}\rest{\Coh(X_1)\cap\Dp(X_1)}\isomor\FM{\ke}\rest{\Coh(X_1)\cap\Dp(X_1)}.
\]
To produce the desired isomorphism
\[
	\theta_3\colon\fun{F}\isomor\FM{\ke}
\]
one argues by induction on the length of the interval to which the non-trivial cohomologies of an object $\kf\in\Dp(X_1)$ belong.

\begin{remark}\label{rmk:CSExtending}
The techniques used to get the
extension $\theta_3$ were improved in \cite{CS2} (see, in
particular, Sections 3.2 and 3.3 of that paper). Indeed, we consider
a wider class of triangulated categories and we deal with extensions
of natural transformations rather than isomorphisms of
functors. These ingredients play a role in the results of Sections
\ref{subsec:abel} and \ref{subsec:supp}.
\end{remark}

\subsection{Exact functors between the abelian categories of coherent sheaves}\label{subsec:abel}

As pointed out in Example \ref{ex:other}, if $X_1$ and $X_2$ are
smooth projective varieties, then the functors induced by exact functors from
$\Coh(X_1)$ to $\Coh(X_2)$ satisfy \eqref{eqn:hyp}, hence Theorem
\ref{thm:CS1} holds for them. This suggests that questions analogous
to (Q1)--(Q5) should be easier to answer for exact functors between
the abelian categories of coherent sheaves. Indeed, for them one can
prove the following result, improving \cite[Prop.\ 5.1]{CS}.

As a matter of notation, if $X_1$ and $X_2$ are smooth projective varieties we denote by
$\kercat{X_1}{X_2}$ the full subcategory of $\Coh(X_1\times X_2)$
having as objects the sheaves $\ke$ which are flat over $X_1$ and
such that $p_2\rest{\supp(\ke)}\colon\supp(\ke)\to X_2$ is a finite
morphism.

\begin{prop}\label{prop:exab}
	
Let $X_1$ and $X_2$ be smooth projective varieties. If $\ke$ is in $\Coh(X_1\times X_2)$, then the additive
functor
\[
\aFM{\ke}:=(p_2)_*(\ke\otimes p_1^*(-))\colon\Coh(X_1)\to\Coh(X_2)
\]
(where $(p_2)_*$ and $\otimes$ are not derived) is exact if and only
if $\ke\in\kercat{X_1}{X_2}$.

Moreover, if we denote by $\ExFun(\Coh(X_1),\Coh(X_2))$ the category
of exact functors from $\Coh(X_1)$ to $\Coh(X_2)$, the functor
\[
\aFM[X_1\to X_2]{\farg}\colon\kercat{X_1}{X_2}\lto
\ExFun(\Coh(X_1),\Coh(X_2))
\]
sending $\ke\in\kercat{X_1}{X_2}$ to $\aFM{\ke}$ is an equivalence of
categories.
\end{prop}

\begin{proof}
We just stick to the second part of the statement and we invite the
reader interested in a proof of the first part to have a look at
\cite{CS}.

We sketch the proof that $\aFM[X_1\to X_2]{\farg}$ is essentially
surjective (again, for more details see \cite{CS}). Hence assume that
$\fun{F}\colon\Coh(X_1)\to\Coh(X_2)$ is an exact functor. By Theorem
\ref{thm:CS1} there exists (unique up to isomorphism)
$\ke\in\Db(X_1\times X_2)$ such that the extension of $\fun{F}$ to the
level of derived categories is isomorphic to $\FM{\ke}$, and
$\ke\in\Coh(X_1\times X_2)$ (to see that $\ke$ is a sheaf, one can
use, for example, \cite[Lemma 2.5]{CS}). From the fact that
$\FM{\ke}(\Coh(X_1))\subseteq\Coh(X_2)$ it is easy to deduce that
$\fun{F}\iso\FM{\ke}\rest{\Coh(X_1)}\iso\aFM{\ke}$.

In order to prove that $\aFM[X_1\to X_2]{\farg}$ is fully faithful,
denoting by $\ks$ the sheaf
$\bigoplus_{m\ge0}(p_2)_*(p_1^*\so_{X_1}(mH))$ of graded algebras on
$X_2$ ($H$ being an ample line bundle on $X_1$), we will use the
relative version of the Serre correspondence between graded $\ks$-modules
and sheaves on $\sproj\ks\iso X_1\times X_2$. More precisely, denoting
by $\gfMod{\ks}$ the category of graded $\ks$-modules of finite type
(meaning finitely generated in sufficiently high degrees), one
considers the associated sheaf functor
$\fun{H}\colon\gfMod{\ks}\to\Coh(X_1\times X_2)$ and the functor
$\fun{G}\colon\Coh(X_1\times X_2)\to\gfMod{\ks}$ defined on objects by
$\fun{G}(\ke):=\bigoplus_{m\in\ZZ}\aFM{\ke}(\so_{X_1}(mH))$. They
satisfy $\fun{H}\comp\fun{G}\iso\id$ and, moreover, an object or a
morphism of $\gfMod{\ks}$ is sent to $0$ by $\fun{H}$ if and only if
it is $0$ in sufficiently high degrees.

Now, given $\ke_1,\ke_2\in\kercat{X_1}{X_2}$, morphisms in
$\gfMod{\ks}$ from $\fun{G}(\ke_1)$ to $\fun{G}(\ke_2)$ can be
identified with natural transformations from
$\aFM{\ke_1}\rest{\cat{C}}$ to $\aFM{\ke_2}\rest{\cat{C}}$, where
$\cat{C}$ is the full subcategory of $\Coh(X_1)$ with objects
$\{\ko_{X_1}(iH)\}_{i\in\ZZ}$. By \cite[Prop.\ 3.6]{CS2} (applied to
the functors $\FM{\ke_1}$ and $\FM{\ke_2}$) such natural
transformations correspond bijectively to natural transformations from
$\aFM{\ke_1}$ to $\aFM{\ke_2}$. Therefore, in view of the properties
of $\fun{G}$ and $\fun{H}$ mentioned above, the fully faithfulness of
$\aFM[X_1\to X_2]{\farg}$ amounts to the following: if
$\alpha\colon\aFM{\ke_1}\to\aFM{\ke_2}$ is a natural transformation
such that $\alpha_m:=\alpha(\ko_{X_1}(mH))=0$ for $m\gg0$, then
$\alpha_m=0$ for every $m\in\ZZ$. Clearly to this purpose it is enough
to show that $\alpha_m=0$ implies $\alpha_{m-1}=0$. To see this, take
a monomorphism $f\colon\ko_{X_1}((m-1)H)\mono\ko_{X_1}(mH)$ and just
observe that in the commutative diagram
\[
\xymatrix{\aFM{\ke_1}(\ko_{X_1}((m-1)H)) \ar[rr]^-{\aFM{\ke_1}(f)}
\ar[d]^{\alpha_{m-1}} & & \aFM{\ke_1}(\ko_{X_1}(mH)) \ar[d]^{\alpha_m=0}
\\
\aFM{\ke_2}(\ko_{X_1}((m-1)H)) \ar[rr]^-{\aFM{\ke_2}(f)} & &
\aFM{\ke_2}(\ko_{X_1}(mH))}
\]
$\aFM{\ke_2}(f)$ is a monomorphism, because $\aFM{\ke_2}$ is exact.
\end{proof}

In particular, this shows that for the functor $\aFM[X_1\to
X_2]{\farg}$ questions (Q1)--(Q4) can be answered positively. As for
(Q5), notice that in general $\kercat{X_1}{X_2}$ is an additive but
not an abelian subcategory of $\Coh(X_1\times X_2)$.

\subsection{The supported case}\label{subsec:supp}

In this section we want to show how Theorem \ref{thm:LO} can be extended both considering a more general categorical setting and weakening the assumptions on the exact functor.

\smallskip

Indeed, let $X$ be a separated scheme of finite type over $\K$ and let $Z$ be
a subscheme of $X$ which is proper over $\K$. We denote by $\D_Z(\Qcoh(X))$ the derived
category of unbounded complexes of quasi-coherent sheaves on $X$ with
cohomologies supported on $Z$. Using this, we can define
the triangulated categories
\begin{equation*}\label{eqn:categories}
\begin{split}
&\Db_Z(\Qcoh(X)):=\D_Z(\Qcoh(X))\cap\Db(\Qcoh(X))\\
&\Db_Z(X):=\D_Z(\Qcoh(X))\cap\Db(X).
\end{split}
\end{equation*}
We also set
\[
\Dp[Z](X):=\D_Z(\Qcoh(X))\cap\Dp(X).
\]

\begin{ex}\label{ex:open}
	These categories appear naturally studying the so called \emph{open Calabi-Yau's}. Examples of them are local resolutions of $A_n$-singularities (\cite{IU,IUU}) and the total space $\mathrm{tot}(\omega_{\PP^2})$ of the canonical bundle of $\PP^2$ (\cite{BaMa}). In the latter case, if $Z$ denotes the zero section of the projection $\mathrm{tot}(\omega_{\PP^2})\to\PP^2$, the derived category $\Dp[Z](\mathrm{tot}(\omega_{\PP^2}))= \Db_{Z}(\mathrm{tot}(\omega_{\PP^2}))$ is a Calabi--Yau category of dimension $3$ and may be seen as an interesting example to test predictions about Mirror Symmetry and the topology of the space of stability conditions according to Bridgeland's definition (see \cite{BaMa} for results in this direction). Moreover, as a consequence of \cite{IU,IUU}, all autoequivalences of the supported derived categories of $A_n$-singularities are of Fourier--Mukai type and the group of such autoequivalences can be explicitly described. See \cite{CS2} for more details.
\end{ex}

The category $\D_Z(\Qcoh(X))$ is a full subcategory of $\D(\Qcoh(X))$ and let
\[
\iota\colon\D_Z(\Qcoh(X))\lto\D(\Qcoh(X))
\]
be the inclusion. This functor has a right adjoint
\[
\iota^!\colon\D(\Qcoh(X))\to\D_Z(\Qcoh(X))\qquad \iota^!(\ke):=\colim\R\sHom(\ko_{nZ},\ke),
\]
where $nZ$ is the $n$-th infinitesimal neighborhood of $Z$ in $X$ (see \cite[Prop.\ 3.2.2]{L}). Due to \cite[Cor.\ 3.1.4]{L}, the functor $\iota^!$ sends bounded complexes to bounded complexes and $\iota^!\comp\iota\iso\id$.

Now, let $X_1$ and $X_2$ be separated schemes of finite type over $\K$ containing, respectively, two subschemes $Z_1$ and $Z_2$ which are proper over $\K$. The following generalizes the standard definition of Fourier--Mukai functor.

\begin{definition}\label{def:FMs}
	An exact functor
	\[
	\fun{F}\colon\D_{Z_1}(\Qcoh(X_1))\to\D_{Z_2}(\Qcoh(X_2))
	\]
	is a \emph{Fourier--Mukai functor} if there exists $\ke\in\D_{Z_1\times Z_2}(\Qcoh(X_1\times X_2))$ and an isomorphism of exact functors
	\begin{equation}\label{eqn:FMdef}
	\fun{F}\iso\FMS{\ke}:=\iota^!(p_2)_*((\iota\times\iota)\ke\otimes p_1^*(\iota(-)))
	\end{equation}
	where $p_i\colon X_1\times X_2\to X_i$ is the projection.
\end{definition}

An analogous definition can be given for functors defined between bounded derived categories of quasi-coherent, coherent or perfect complexes. As always, the object $\ke$ is called \emph{Fourier--Mukai kernel}. It should be noted that, contrary to the smooth non-supported case, the Fourier--Mukai kernel cannot be assumed to be a bounded coherent complex. This is clarified by the following example dealing with the identity functor.

\begin{ex}\label{ex:identity}
We want to show that a Fourier--Mukai kernel of the identity functor
$\id\colon\Db_{Z}(X)\to\Db_{Z}(X)$ is
\[
(\iota\times\iota)^!\ki\in
\Db_{Z\times Z}(\Qcoh(X\times X)),
\]
where, denoting by $\Delta\colon X\to X\times X$ the diagonal embedding,
	\[
	\ki:=\Delta_*\comp\iota\comp\iota^!(\ko_X).
	\]
Indeed, according to \cite{CS2}, we have the following isomorphisms:
\begin{equation*}
\begin{split}
\Hom(\ka,\iota^!\FM{\ki}(\iota\kb))&\iso\Hom(\iota\ka,(p_2)_*(\Delta_*\comp\iota\comp\iota^!(\ko_X)\otimes p_1^*(\iota\kb)))\\
&\iso\Hom((\iota\kb)\dual\boxtimes\iota\ka,\Delta_*\comp\iota\comp\iota^!(\ko_X))\\
&\iso\Hom((\iota\kb)\dual\otimes\iota\ka,\ko_X)\\
&\iso\Hom(\ka,\iota^!\iota\kb)\iso\Hom(\ka,\kb),
\end{split}
\end{equation*}
for any $\ka,\kb\in\Db_{Z}(X)$. Here $p_i\colon X\times X\to X$ is the natural projection. For the first and the fourth isomorphism we used the adjunction between $\iota$ and $\iota^!$. The same adjunction together with the one between $\Delta^*$ and $\Delta_*$ and the fact that $\iota$ is fully faithful and $(\iota\kb)\dual\otimes\iota\ka$ has support in $Z$ explains the third isomorphism.

Obviously $(\iota\times\iota)^!\ki$ does not belong to $\Db_{Z\times Z}(X\times X)$. Suppose that there
exists $\ke\in\Db_{Z\times Z}(X\times X)$ such that
\[\FMS{\ke}\iso\id\colon\Db_Z(X)\to\Db_Z(X).\]
By \cite[Lemma 7.41]{R}, there exist $n>0$ and $\ke_n\in\Db(nZ\times nZ)$ such that
$(\iota\times\iota)\ke\iso(i_n\times i_n)_*\ke_n$, where $i_n\colon
nZ\to X$ is the embedding. For any $\kf_n\in\Db(nZ)$, we have
\begin{equation}\label{eqn:split1}
(i_n)_*\kf_n\iso\FM{\ke}((i_n)_*\kf_n)\iso(i_n)_*\FM{\ke_n}((i_n)^*(i_n)_*\kf_n).
\end{equation}

Take now $X=\PP^k$, $Z=\PP^{k-1}$ and $\kf_n:=\ko_{nZ}(m)$, for $m\in\ZZ$. An easy calculation shows that $(i_n)^*(i_n)_*\kf_n\iso\ko_{nZ}(m)\oplus\ko_{nZ}(m-n)[1]$. Hence to have \eqref{eqn:split1} verified, we should have either $\FM{\ke_n}(\ko_{nZ}(m))=0$ or $\FM{\ke_n}(\ko_{nZ}(m-n))=0$. But the following isomorphisms should hold at the same time
\[
\begin{split}
&\FM{\ke_n}(\ko_{nZ}(m))\oplus\FM{\ke_n}(\ko_{nZ}(m-n))[1]\iso\ko_{nZ}(m),\\
&\FM{\ke_n}(\ko_{nZ}(m+n))\oplus\FM{\ke_n}(\ko_{nZ}(m))[1]\iso\ko_{nZ}(m+n).
\end{split}
\]
If $\FM{\ke_n}(\ko_{nZ}(m-n))=0$, then from the second one we would have that $\ko_{nZ}(m)[1]$ is a direct summand of $\ko_{nZ}(m+n)$ which is absurd. Thus $\FM{\ke_n}(\ko_{nZ}(m))=0$. As this holds for all $m\in\ZZ$, we get a contradiction.
\end{ex}

\smallskip

Now let $X_1$ be a quasi-projective scheme containing a projective subscheme $Z_1$ such that $\ko_{iZ_1}\in\Dp(X_1)$, for all $i>0$, and let $X_2$ be a separated scheme of finite type over $\K$ with a subscheme $Z_2$ which is proper over $\K$.

\begin{remark}\label{rmk:amplein}
Notice that under these assumptions, and having fixed an ample divisor $H_1$
on $X_1$, the objects $\ko_{|i|Z_1}(jH_1)$ are in $\Dp[Z_1](X_1)$, for all $i,j\in\ZZ$. Special cases in which $\ko_{iZ_1}\in\Dp(X_1)$ are when $X_1=Z_1$ or $X_1$ is smooth.
\end{remark}

    One can consider exact functors $\fun{F}\colon\Dp[Z_1](X_1)\to\Dp[Z_2](X_2)$ such that
	\medskip
	\begin{itemize}
	\item[$\Cc$]
	\begin{itemize}
	\item[(1)] \emph{$\Hom(\fun{F}(\ka),\fun{F}(\kb)[k])=0$, for
          any $\ka,\kb\in\Coh_{Z_1}(X_1)\cap\Dp[Z_1](X_1)$ and any
          integer $k<0$;
	\item[(2)] For all $\ka\in\Dp[Z_1](X_1)$ with trivial
          cohomologies in (strictly) positive degrees, there is
          $N\in\ZZ$ such that
          \[\Hom(\fun{F}(\ka),\fun{F}(\ko_{|i|Z_1}(jH_1)))=0,\]
          for any $i<N$ and any $j\ll i$, where $H_1$ is an ample
          divisor on $X_1$.}
	\end{itemize}
	\end{itemize}
	\medskip
Then we have the following.

\begin{thm}\label{thm:CS2}{\bf(\cite{CS2}, Theorem 1.1.)}
	Let $X_1$, $X_2$, $Z_1$ and $Z_2$ be as above and let
\[
\fun{F}\colon\Dp[Z_1](X_1)\lto\Dp[Z_2](X_2)
\]
be an exact functor.

If $\fun{F}$ satisfies $\Cc$, then there exist $\ke\in\Db_{Z_1\times Z_2}(\Qcoh(X_1\times X_2))$ and an isomorphism of exact functors $\fun{F}\iso\FMS{\ke}$. Moreover, if $X_i$ is smooth quasi-projective, for $i=1,2$, and $\K$ is perfect, then $\ke$ is unique up to isomorphism.
\end{thm}

Back to Remark \ref{rmk:amplein}, the above theorem can be applied in at least two interesting geometric contexts. If $X_1=Z_1$, then we get back (a generalization of) Theorem \ref{thm:LO}. On the other hand, if $X_1$ is smooth, then we can apply the above result to the autoequivalences of the categories described in Example \ref{ex:open} proving that they are all of Fourier--Mukai type. As noticed in \cite{CS2}, if $X_i=Z_i$, $\dim(X_1)>0$ and they are smooth, then
$\Cc$ is equivalent to \eqref{eqn:hyp}. Thus, Theorem \ref{thm:CS2} recovers Theorem \ref{thm:CS1} as
well.

\begin{remark}\label{rmk:suppuniqen}
	In the same vein as in \cite{LO}, it is proved in \cite[Thm.\ 1.2]{CS2} that $\Dp[Z](X)$ has a (strongly) unique dg-enhancement if $X$ and $Z$ have the same properties as $X_1$ and $Z_1$ in Theorem \ref{thm:CS2} and $T_0(\ko_Z)=0$. See \cite{LO} for the definition of strongly unique dg-enhancement which is not needed here.
\end{remark}

\section{More open problems}\label{sec:openproblems}

The list of problems mentioned in the above sections can be extended further. The main sources are actually very concrete geometric settings where they appear naturally. We try to list some of them below, although a complete clarification of their geometric meaning goes far beyond the scope of this paper.

\subsection{Does full imply essentially surjective?}\label{subsec:fullauto}

In Section \ref{subsec:full} we have seen that a full functor between the bounded derived categories of coherent sheaves on smooth projective varieties is automatically faithful. Assume now that we are given an exact endofunctor $\fun{F}\colon\Db(X)\to\Db(X)$, where $X$ is again a smooth projective variety. In this section we want to discuss the following.

\begin{conjecture}\label{conj:fullequiv}
If $\fun{F}$ is full, then it is an autoequivalence.
\end{conjecture}

Notice that we only need to show that $\fun{F}$ is essentially
surjective.

\begin{remark}\label{rmk:cantriv}
The conjecture is true if $\ds_X$ is trivial, because in that case
every fully faithful exact endofunctor of $\Db(X)$ is an equivalence
(see, for example, \cite[Cor.\ 7.8]{H}).
\end{remark}

The above conjecture is implied by another conjecture about
admissible subcategories that we want to explain here.

\medskip

Given a triangulated category $\cat{T}$ and a strictly full
triangulated subcategory $\cat{S}$, we say that $\cat{S}$ is
\emph{left-} (resp.\ \emph{right-}) \emph{admissible} in $\cat{T}$ if
the inclusion functor $\eta\colon\cat{S}\to\cat{T}$ has a left
(resp.\ right) adjoint $\eta^*\colon\cat{T}\to\cat{S}$
(resp.\ $\eta^!\colon\cat{T}\to\cat{S}$). If a subcategory is left
and right admissible, we say that it is \emph{admissible}.

\begin{remark}\label{rmk:leftright}
	By \cite[Prop.\ 1.6]{BK}, an admissible subcategory $\cat{S}\subseteq\cat{T}$ is thick as well.
\end{remark}

We can use the notion of admissible subcategory to `decompose'
triangulated categories. More generally, one can give the following.

\begin{definition}\label{def:semiorth}
A \emph{semi-orthogonal decomposition} of a triangulated category $\cat{T}$ is given by a sequence of full triangulated subcategories $\cat{A}_1,\ldots,\cat{A}_n\subseteq\cat{T}$ such that $\Hom_{\cat{T}}(\cat{A}_i,\cat{A}_j)=0$, for $i>j$ and, for all $K\in\cat{T}$, there exists a chain of morphisms in $\cat{T}$
\[
0=K_n\to K_{n-1}\to\ldots\to K_1\to K_0=K
\]
with $\cone{K_i\to K_{i-1}}\in\cat{A}_i$, for all $i=1,\ldots,n$.
We will denote such a decomposition by $\cat{T}=\ort{\cat{A}_1,\ldots,\cat{A}_n}$.
\end{definition}

The easiest examples of semi-orthogonal decompositions are constructed via exceptional objects.

\begin{definition}\label{def:excobjgeneral}
Assume that $\cat{T}$ is a $\K$-linear triangulated category. An object $E\in\cat{T}$ is called \emph{exceptional} if $\Hom_{\cat{T}}(E,E)\iso\K$ and $\Hom_{\cat{T}}(E,E[p])=0$, for all $p\neq0$.
A sequence $(E_1,\ldots,E_m)$ of objects in $\cat{T}$ is called an \emph{exceptional sequence} if $E_i$ is an exceptional object, for all $i$, and $\Hom_{\cat{T}}(E_i,E_j[p])=0$, for all $p$ and all $i>j$. An exceptional sequence is \emph{full} if it generates $\cat{T}$.
\end{definition}

\begin{remark}\label{rmk:semexc}
If $(E_1,\ldots,E_m)$ is a full exceptional sequence in $\cat{T}$, then we get a semi-orthogonal decomposition $\cat{T}=\ort{E_1,\ldots,E_n}$, where for simplicity we write $E_i$ for the triangulated subcategory generated by $E_i$, which is equivalent to $\Db(\spec\K)$ and is admissible in $\cat{T}$.
\end{remark}

\begin{ex}\label{ex:Pn}
A celebrated result of Beilinson shows that $\Db(\PP^n)$ has a full exceptional sequence $(\ko_{\PP^n}(-n),\ko_{\PP^n}(-n+1),\ldots,\ko_{\PP^n})$ (see, for example, \cite[Sect.\ 8.3]{H}).
\end{ex}

For a triangulated subcategory $\cat{S}$ of a triangulated category $\cat{T}$, we can define the strictly full triangulated subcategories (i.e.\ full and closed under isomorphism)
\[
\rort{\cat{S}}:=\left\{A\in\cat{T}:\Hom(S,A)=0,\text{ for all }S\in\cat{S}\right\}
\]
called \emph{right orthogonal} to $\cat{S}$ and its \emph{left orthogonal}
\[
\lort{\cat{S}}:=\left\{A\in\cat{T}:\Hom(A,S)=0,\text{ for all }S\in\cat{S}\right\}.
\]

One can formulate the following conjecture due to A.\ Kuznetsov and contained in \cite{K4}.

\begin{conjecture}\label{conj:noeth} {\bf (Noetherianity conjecture)}
	Let $X$ be a smooth projective variety and assume that there exists a sequence
\[
\cat{A}_1\subseteq\cat{A}_2\subseteq\ldots\subseteq\cat{A}_i\subseteq\ldots\subseteq\Db(X)
\]
of admissible subcategories. Then there is a positive integer $N$ such that $\cat{A}_i=\cat{A}_N$, for all $i\geq N$.
\end{conjecture}

\begin{remark}\label{rmk:otherdirect}
	Considering the strictly full triangulated subcategories $\cat{B}_i:=\rort{\cat{A}_i}$, the above conjecture can be equivalently reformulated in terms of stabilizing descending chains.
\end{remark}

\begin{prop}\label{prop:conjimplconj}
	Conjecture \ref{conj:noeth} implies Conjecture \ref{conj:fullequiv}.
\end{prop}

\begin{proof}
The functor $\fun{F}$ is automatically faithful. Thus
\[
\cat{I}:=\im\fun{F}:=\{\ke\in\Db(X):\ke\iso\fun{F}(\kf)\text{ for some } \kf\in\Db(X)\}
\]
is a strictly full triangulated subcategory of $\Db(X)$. By Proposition \ref{prop:exadj}, the functor $\fun{F}$ has left and right adjoints and so $\cat{I}$ is admissible. Using the above notation, set $\cat{J}=\rort{\cat{I}}$. Hence we have a semi-orthogonal decomposition
\[
\Db(X)=\ort{\cat{J},\cat{I}}.
\]

As $\cat{I}\iso\Db(X)$, we can think of $\fun{F}$ as an exact endofunctor of $\cat{I}$. Hence, reasoning as above we get a semi-orthogonal decomposition
\[
\Db(X)=\ort{\cat{J},\cat{J},\cat{I}}
\]
Hence, given a positive integer $n$, repeating this argument $n$ times we get that
\[
\cat{A}_n:=\ort{\underbrace{\cat{J},\ldots,\cat{J}}_{n \text{times}}}
\]
is a strictly full admissible triangulated subcategory of $\Db(X)$.

Since $\cat{A}_p\subseteq\cat{A}_q\subseteq\Db(X)$ whenever $p\leq q$, by Conjecture \ref{conj:noeth}, this sequence must stabilize. Hence $\cat{J}=0$ and so $\fun{F}$ is essentially surjective.
\end{proof}

Due to the following easy result, a full endofunctor is automatically an equivalence when $X$ has dimension at most $1$.

\begin{prop}\label{prop:curves}
	Conjecture \ref{conj:noeth} holds true when $X$ is a smooth projective variety of dimension smaller or equal to $1$.
\end{prop}

\begin{proof}
Obviously the Conjecture is trivially true if $\Db(X)$ does not admit a non-trivial semi-orthogonal decomposition and this is the case if $\dim(X)=0$.

If $X$ is a curve of genus $1$, Serre duality and \cite[Example 3.2]{Br} implies that $\Db(X)$ cannot be decomposed. The same is true when $X$ is a curve of genus $g\geq 2$ due to \cite{Ok}. Thus the only case that has to be checked is $X\iso\PP^1$.

For this assume that $\Db(\PP^1)=\ort{\cat{A}_1, \cat{A}_2}$, where $\cat{A}_i$ is not trivial, for $i=1,2$ (i.e.\ non-zero and not the whole category $\Db(\PP^1)$). It is clear that either $\cat{A}_1$ or $\cat{A}_2$ must contain a locally free sheaf $E$. As on $\PP^1$ any locally free sheaf is the direct sum of line bundles and $\cat{A}_i$ is thick (see Remark \ref{rmk:leftright}(i)), there is $j\in\ZZ$ such that $\ko_{\PP^1}(j)\in\cat{A}_i$, for $i=1$ or $i=2$. We assume $i=1$ as the argument in the other case is similar.

Now $\cat{A}_2=\lort{\cat{A}_1}\subseteq\lort{\ort{\ko_{\PP^1}(j)}}=\ort{\ko_{\PP^1}(j+1)}$. But $\ort{\ko_{\PP^1}(j+1)}\iso\Db(\spec{\K})$ and so it does not contain proper thick subcategories. Thus $\cat{A}_2=\ort{\ko_{\PP^1}(j+1)}$ and $\cat{A}_1=\ort{\ko_{\PP^1}(j)}$. Therefore, there cannot be non-stabilizing ascending chains of admissible subcategories.
\end{proof}

\subsection{Splitting functors}\label{subsec:splitting}

Kuznetsov introduced in \cite{K2} the notion of splitting functor as a natural generalization of fully faithful functor. The expectation was that, in this context, one should get a representability result similar to Theorem \ref{thm:Orlov}. Let us clarify the situation a bit more.

\medskip

More precisely, given two triangulated categories $\cat{T}_1$ and $\cat{T}_2$ and an exact functor
$\fun{F}\colon\cat{T}_1\to\cat{T}_2$, we can define the following full subcategories
\[
\ker\fun{F}:=\{A\in\cat{T}_1:\fun{F}(A)\iso0\}\qquad\im\fun{F}:=\{A\iso\fun{F}(B):B\in\cat{T}_1\}.
\]

\begin{remark}\label{rmk:kerim}
The subcategory $\ker\fun{F}$ is always triangulated while $\im\fun{F}$, in general, is not. It becomes triangulated if $\fun{F}$ is fully faithful.
\end{remark}

Hence we can give the following.

\begin{definition}\label{def:splitting}
An exact functor $\fun{F}\colon\cat{T}_1\to\cat{T}_2$ is \emph{right}
(respectively \emph{left}) \emph{splitting} if $\ker\fun{F}$ is a
right (respectively left) admissible subcategory in $\cat{T}_1$, the
restriction of $\fun{F}$ to $\rort{(\ker\fun{F})}$ (respectively
$\lort{(\ker\fun{F})}$) is fully faithful, and the category
$\im\fun{F}$ is right (respectively left) admissible in $\cat{T}_2$.

An exact functor is \emph{splitting} if it is both right and left
splitting.
\end{definition}

\begin{remark}\label{adjspl}
As observed in \cite[Lemma 3.2]{K2}, a right (respectively left) splitting functor $\fun{F}$ has a right (respectively left) adjoint functor $\fun{F}^!$ (respectively $\fun{F}^*$).
\end{remark}

We summarize the basic properties of these functors in the following.

\begin{thm}\label{thm:splittprop}{\bf(\cite{K2}, Theorem 3.3.)}
Let $\fun{F}\colon\cat{T}_1\to\cat{T}_2$ be an exact functor.
Then the following conditions are equivalent:
\begin{itemize}
\item[{\rm (i)}] $\fun{F}$ is right splitting;

\item[{\rm (ii)}] $\fun{F}$ has a right adjoint functor $\fun{F}^!$
and the composition of the canonical morphism of functors
$\id_{\cat{T}_1}\to\fun{F}^!\comp\fun{F}$ with $\fun{F}$ gives an isomorphism
$\fun{F}\iso\fun{F}\comp\fun{F}^!\comp\fun{F}$;

\item[{\rm (iii)}] $\fun{F}$ has a right adjoint functor $\fun{F}^!$,
there are semi-orthogonal decompositions
\[
\cat{T}_1=\ort{\im\fun{F}^!,\ker\fun{F}},
\qquad
\cat{T}_2=\ort{\ker\fun{F}^!,\im\fun{F}},
\]
and the functors $\fun{F}$ and $\fun{F}^!$ give quasi-inverse equivalences
$\im\fun{F}^!\iso\im\fun{F}$;

\item[{\rm (iv)}] There exists a triangulated category $\cat{S}$ and fully faithful
functors $\fun{G}_1\colon\cat{S}\to\cat{T}_1$, $\fun{G}_2\colon\cat{S}\to\cat{T}_2$, such that $\fun{G}_1$ admits a left adjoint $\fun{G}_1^*$, $\fun{G}_2$ admits a right adjoint and $\fun{F}\iso\fun{G}_2\comp\fun{G}_1^*$.
\end{itemize}
\end{thm}

Clearly, one can formulate analogous conditions for left splitting functors. The main conjecture is now the following:

\begin{conjecture}\label{conj:splitt}{\bf(\cite{K2}, Conjecture 3.7.)}
Let $X_1$ and $X_2$ be smooth projective varieties. Then any exact splitting functor $\fun{F}\colon\Db(X_1)\to\Db(X_2)$ is of Fourier--Mukai type.
\end{conjecture}

One may first wonder why the strategy outlined in Section \ref{subsec:ingredients} may not be applied in this case. The main problem is that convolutions do not work for this kind of functors. Alternatively, one would need to define an analogue of the ample sequence in Section \ref{subsubsec:ample} for the subcategory $\cat{S}$ in part (iv) of Theorem \ref{thm:splittprop}. Hence, the solution to Conjecture \ref{conj:splitt} is closely related to Problems \ref{prob:ample} and \ref{prob:conv}.

Nevertheless, there are several instances in which the conjecture is verified. The easiest one is when the category $\cat{S}$ mentioned in Theorem \ref{thm:splittprop}(iv) is such that $\cat{S}\iso\Db(Y)$, for some smooth projective variety $Y$. Indeed, in this case, one reduces the proof to Theorem \ref{thm:Orlov} (using Proposition \ref{prop:FMpropert}).

Moreover, it is not difficult to observe that, using the same type of arguments as in the proof of Proposition \ref{prop:curves}, one can show the following (the zero-dimensional case is trivial).

\begin{prop}\label{prop:splittingcurves}
Let either $X_1$ or $X_2$ be a smooth projective curve. Then any splitting functor $\fun{F}\colon\Db(X_1)\to\Db(X_2)$ is of Fourier--Mukai type.
\end{prop}

\smallskip

For less trivial situations where Conjecture \ref{conj:splitt} can be verified, one has to refer to \cite{K3}. For this consider a full admissible subcategory $\eta\colon\cat{S}\mono\Db(X)$, for a smooth projective variety $X$. Thus we get the left and right adjoints $\eta^*\colon\Db(X)\to\cat{S}$ and $\eta^!\colon\Db(X)\to\cat{S}$.

Take now the functors $\fun{F}_1:=\eta\comp\eta^!\colon\Db(X)\to\Db(X)$ and $\fun{F}_2:=\eta\comp\eta^*\colon\Db(X)\to\Db(X)$. It is not difficult to see (using, for example, Theorem \ref{thm:splittprop} above) that $\fun{F}_1$ and $\fun{F}_2$ are splitting functors. A non-trivial argument allows one to prove the following:

\begin{thm}\label{thm:proj}{\bf(\cite{K3}, Theorem 7.1.)}
The functors $\fun{F}_1$ and $\fun{F}_2$ are of Fourier--Mukai type.
\end{thm}

\subsection{Relative Fourier--Mukai functors}\label{subsec:relative}

In \cite{K2}, Kuznetsov drove the attention to a slightly more general version of the classical Fourier--Mukai functors. For sake of simplicity, take a pair of smooth projective varieties $X_1$ and $X_2$ over the same smooth projective variety $S$. To fix the notation, this means that, for $i=1,2$, there is a morphism $f_i\colon X_i\to S$. Clearly, one may want to relax the assumptions on $X_i$ and $S$ but this is not in order here.

\begin{definition}\label{def:relfun}
	{\rm (i)} A functor $\fun{F}\colon\Db(X_1)\to\Db(X_2)$ is \emph{$S$-linear} if
	\[
	\fun{F}(\ka\otimes f_1^*(\kc))\iso\fun{F}(\ka)\otimes f_2^*(\kc),
	\]
	for all $\ka\in\Db(X_1)$ and for all $\kc\in\Db(S)$.

{\rm (ii)} 	A strictly full subcategory $\cat{S}\subseteq\Db(X_i)$ is \emph{$S$-linear} if for all $\kc\in\cat{S}$ and all $\ka\in\Db(S)$ we have $f_i^*(\ka)\otimes\kc\in\cat{S}$.
\end{definition}

These functors have reasonable properties listed in the following proposition and proved in \cite{K2} (see, in particular, Section 2.7 there).

\begin{prop}\label{prop:proprrel}
	{\rm (i)} If $\fun{F}$ is exact, $S$-linear and admits a right adjoint functor $\fun{F}^!$, then $\fun{F}^!$ is also $S$-linear.
	
	{\rm (ii)} If $\cat{S}\subseteq\Db(X_i)$ is a strictly full admissible $S$-linear subcategory, then its right and left orthogonals are $S$-linear as well.
\end{prop}

As pointed out in, for example, \cite{K3,K2}, the relative functors play important roles in various geometric situations. Thus it makes perfect sense to wonder whether the machinery developed for Fourier--Mukai functors in the non-relative setting can be applied.

\medskip

It is clear that any full exact $S$-linear functor or rather any exact $S$-linear functor $\fun{F}\colon\Db(X_1)\to\Db(X_2)$ satisfying \eqref{eqn:hyp} is of Fourier--Mukai type in view of Theorem \ref{thm:CS1}. In particular, there is a unique (up to isomorphism) $\ke\in\Db(X_1\times X_2)$ and an isomorphism $\fun{F}\iso\FM{\ke}$.

On the other hand, we may consider the fibre product $X_1\times_S X_2$ and the closed embedding $i\colon X_1\times_S X_2\mono X_1\times X_2$.

\begin{lem}\label{lem:lemrel}{\bf (\cite{K2}, Lemma 2.32.)}
	If $\ke\in\Db(X_1\times_S X_2)$, then the Fourier--Mukai functor $\FM{i_*\ke}$ is $S$-linear.
\end{lem}

It is not difficult to observe that the Fourier--Mukai kernel of an $S$-linear Fourier--Mukai functor has to be set theoretically supported on the fibre product $X_1\times_S X_2$. The scheme theoretical point of view is more complicated to be dealt with and thus, following \cite{K2}, it makes sense to pose the following questions:

\begin{qn}\label{qn:rel}
	{\rm (i)} Given a full exact $S$-linear functor $\fun{F}\colon\Db(X_1)\to\Db(X_2)$, do there exist an $\ke\in\Db(X_1\times_S X_2)$ and an isomorphism of functors $\fun{F}\iso\FM{i_*\ke}$?
	
	{\rm (ii)} Is the choice of the Fourier--Mukai kernel $\ke\in\Db(X_1\times_S X_2)$ in (i) unique (up to isomorphism)?
\end{qn}

To our knowledge, no general answer to these problems is present in the literature.

%%%%%%%%%%%%%%%%%%%%%%%%%%%%%%%%%%%%%

\bigskip

{\small\noindent {\bf Acknowledgements.} The write-up of this paper started when P.S.\ was visiting the University of Bonn which we thank for the warm hospitality and for the financial support. The second author is also grateful to the organizers of the GCOE Conference ``Derived Categories 2011 Tokyo'', Y.\ Kawamata and Y.\ Toda, for the very stimulating mathematical atmosphere during the conference. We are also grateful to Pawel Sosna for comments on an early version of this paper. Pierre Schapira informed us about the paper \cite{SKK} and the notion of Fourier--Sato transform. David Ben--Zvi kindly brought our attention to the results in \cite{BZFN} and \cite{Pr}. We warmly thank both of them.}

%%%%%%%%%%%%%%%%%%%%%%%%%%%%%%%%%%%%%

\end{document}